\documentclass[11pt, a4paper, twoside]{article}
\usepackage[a4paper, margin=3.75cm]{geometry}
\usepackage{amsmath}
\usepackage{amssymb}
\usepackage{amsthm}
\usepackage{mathrsfs}
\usepackage{graphicx}

\usepackage{titlesec}
\titleformat{\section}{\normalfont\Large\bfseries}{\thesection.}{0.5em}{}
\titleformat{\subsection}[runin]{\normalfont\normalsize\bfseries}{\thesubsection.}{0.5em}{}
\titleformat{\subsubsection}[runin]{\normalfont\normalsize\bfseries}{\thesubsubsection.}{0.5em}{}

\usepackage{fancyhdr}
\usepackage{hyperref}
\hypersetup{
  unicode=true,
  pdfborder={0 0 0},
  pdffitwindow=false,
  pdfstartview={FitH},
  pdfpagelayout={OneColumn},
  pdftitle={On Convergence of Rational Hermite-Padé Approximants},
  pdfauthor={Nikolay R. Ikonomov, Sergey P. Suetin},
  pdfsubject={arXiv},
  pdfkeywords={analytic extension, Padé approximants, Hermite-Padé polynomials},
  allcolors=blue,
  colorlinks=true,
  breaklinks=true
}

\def\({\left(}
\def\){\right)}
\def\[{\left[}
\def\]{\right]}
\def\<{\left\langle}
\def\>{\right\rangle}

\let\leq\leqslant
\let\geq\geqslant

\let\myh\widehat
\let\myt\widetilde
\let\myo\overline
\let\eps\varepsilon

\def\NN{\mathbb N}
\def\RR{\mathbb R}
\def\CC{\mathbb C}
\def\PP{\mathbb P}
\def\HH{\mathscr H}

\def\supp{\operatorname{supp}}
\def\mcap{\operatorname{cap}}
\def\mdeg{\operatorname{deg}}
\def\const{\operatorname{const}}
\def\meas{\operatorname{meas}}
\def\equ{\operatorname{equ}}
\def\GRS{\operatorname{GRS}}

\theoremstyle{plain}
\newtheorem{theorem}{Theorem}
\newtheorem{lemma}{Lemma}
\newtheorem{corollary}{Corollary}

\theoremstyle{definition}
\newtheorem{remark}{Remark}

\begin{document}

\title{\Large\bfseries On Convergence of Rational Hermite--Padé Approximants}

\author{Nikolay R. Ikonomov${}^1$, Sergey P. Suetin${}^2$\\[2mm]
\normalsize ${}^1$Institute of Mathematics and Informatics,\\
\normalsize Bulgarian Academy of Sciences, Bulgaria\\
\normalsize \href{mailto:nikonomov@math.bas.bg}{nikonomov@math.bas.bg}\\[2mm]
\normalsize ${}^2$Steklov Mathematical Institute,\\
\normalsize Russian Academy of Sciences, Russia\\
\normalsize \href{mailto:suetin@mi-ras.ru}{suetin@mi-ras.ru}}

\date{ }
\maketitle

\begin{abstract}
The main purpose of this paper is to compare the convergence properties of Padé approximants and rational Hermite--Padé approximants for some model class of multivalued analytic functions based of the inverse Zhoukovsky transform. We prove that in the class of analytic functions under consideration the rational Hermite--Padé approximants converge faster than the corresponding Padé approximants.

In contrast  to the classical vector potential-theoretic approach, which was introduced by A. A. Gonchar and E. A. Rakhmanov in 1981 and developed later by A. I. Aptekarev, V. N. Sorokin and others, the proofs here are based on some scalar mixed Green-logarithmic potential problems.

Bibliography: \cite{Tre23} titles.

\end{abstract}

\footnotetext[0]{The work of the second author was performed at the Steklov International Mathematical Center and supported by the Ministry of Science and Higher Education of the Russian Federation (agreement no. 075-15-2025-303).}

\section{Introduction and Statement of the Problem}\label{s1}

\subsection{}\label{s1s1}
Let a function $f(z)\in\HH(\infty)$ be given by the following explicit representation:
\begin{equation}
f(z):=\[\(A-\frac1{\varphi(z)}\)\(B-\frac1{\varphi(z)}\)\]^{-1/2},
\quad z\notin[-1,1],
\label{funzhu}
\end{equation}
where $1<A<B$, $\varphi(z)=z+(z^ 2-1)^{1/2}$, and such branch of the function $(\cdot)^{1/2}$ is chosen that $\varphi(z)\sim 2z$ and $f(z)\sim1/\sqrt{AB}>0$ as $z\to\infty$. The function $f$ is an algebraic function of the fourth order with exactly four branch points $\{\pm1,a,b\}$ of square root type each, and with $a=(A+1/A)/2$ and $b=(B+1/B)/2$, $1<a<b$. The segment $ E: = [-1,1]$  is the Stahl compact set  $S (f)$ for the given function \eqref {funzhu} and the domain $D:=\myh{\CC}\setminus{E}$ is the corresponding Stahl domain \cite{Sta97b}. We  denote this class of analytic functions given by \eqref{funzhu}  with arbitrary $1<A<B$ over $\mathscr Z (E)$.

In the paper \cite {Sue18c} was proven that each function  $f$ from the class $\mathscr Z(E)$ is a Markov type function  and  the corresponding pair $ f, f ^ 2$, as well as the triple $ f, f^2,f^3$, form Nikishin systems  (see \cite{NiSo88}, \cite{ApLy21}):
\begin{align}
f(z)&=\frac1{\sqrt{AB}}+\myh{\sigma}(z), & \supp{\sigma}&=E,\label {repf}\\
f^2(z)&=\frac1{AB}+\frac1{\sqrt{AB}}\myh{\sigma}(z) +\myh{s}_1(z), & \supp{s_1}&=E,\label {repf2}\\
f^3(z)&=\frac1{\sqrt{(AB)^3}}+\frac1{AB}\myh{\sigma}(z)+\frac1{\sqrt{AB}}\myh{s}_1(z)+\myh{s}_2(z), & \supp{s_2}&=E,
\label {repf3}
\end{align}
where $\supp\sigma=\supp{s}_1=\supp{s}_2=E$, $s_1=\<\sigma,\sigma_2\>$, i.e., $ds_1(x):=\myh{\sigma}_2(x)\,d\sigma(x)$, $\supp{\sigma_2}=[a,b]$, and $s_2:=\<\sigma,\<\sigma_2,\sigma\>\>$. The measures $\sigma$ and $\sigma_2$ both admit  explicit representations, see \cite[equations (16)--(17)]{Sue18c}.
Thus the function $f\in\mathscr Z(E)$ is a holomorphic function in the domain $D=\myh{\CC}\setminus E$. The same is valid for the functions $f^2$ and $f^3$.

\fancyhf{}
\fancyhead[LE]{\small\thepage}
\fancyhead[CE]{\small N. R. Ikonomov, S. P. Suetin}
\fancyhead[RO]{\small\thepage}
\fancyhead[CO]{\small On Convergence of Rational Hermite--Padé Approximants}

Recall that for a positive Borel measure 
 $\varkappa$, $\supp{\varkappa}\Subset\RR$, we denote by
\begin{equation}
\myh{\varkappa}(z):=\int\frac{d\varkappa(x)}{z-x},\quad z\in\myh{\CC}\setminus\supp{\varkappa},
\label{Marfun}
\end{equation}
the corresponding Markov type function, from now on we shall assume that in the representations of type \eqref{Marfun} the measure $\varkappa$ is such that $\varkappa\,'(x)=d\varkappa(x)/dx>0$ almost everywhere (a.e.) on $\supp{\varkappa}$.

Also recall (see \cite{GoRaSo97}) that for two measures $\varkappa$ and $\nu$, $\supp{\varkappa}, \supp{\nu}\Subset\RR$, 
$\supp{\varkappa}\cap\supp{\nu}=\varnothing$, it is defined
$$
d\<\varkappa,\nu\>(x):=\myh{\nu}(x)\,d\varkappa(x),
\quad
\myh{\<\varkappa,\nu\>}(z):=
\int\frac{\myh{\nu}(x)\,d\varkappa(x)}{z-x}.
$$

For an arbitrary compact set  $K\subset\RR$, with capacity $\mcap{K}>0$,   denote by  $M_1(K)$ the set of all probability measures supported on $K$. For a compact set  $S\subset\RR$ denote by $g_S(\zeta, z)$  the Green function for the domain $\myh {\CC}\setminus{S}$  with logarithmic singularity at the point $\zeta=z$.   Let $V ^\mu(z)$ be logarithmic potential of a measure $\mu\in M_1(K)$ and $G ^\mu_S(z)$  be the Green potential of the measure $\mu\in M_1(K)$ with respect to the Green function  $g_S(\zeta,z)$, $S\cap K=\varnothing$:
\begin{equation}
V^\mu(z):=\int_K\log\frac1{|z-t|}\,d\mu(t),\quad
G^\mu_S(z):=\int_K g_S(t,z)\,d\mu(t).
\label{pot}
\end{equation}

It is well known \cite{Lan66}, \cite{SaTo97} that if $K$ is a regular compact set, $\mcap{K}>0$, then there exists a unique probability  measure $\tau ^{\vphantom{p}}_K\in M_1(K)$ with the following property:
\begin{equation}
V^{\tau^{\vphantom{p}}_K}(x)\equiv\gamma^{\vphantom{p}}_K=\const,\quad x\in K.
\label{equrob}
\end{equation}
The measure $\tau^{\vphantom{p}}_K$ is referred to as Robin measure or equilibrium measure for the compact set $K$, the value  $\gamma^{\vphantom{p}}_K$ is called the Robin constant for the compact set $K$. Note that $\supp{\tau_K}=K$. Let  
$$
g_K(z,\infty)=\log|z|+\gamma^{\vphantom{p}}_K+o(1),\quad z\to\infty,
$$
be the Green function for the domain $D =D(K):=\myh{\CC}\setminus{K}$ with the logarithmic singularity at the infinity point $z=\infty$. Then we have
\begin{equation}
g_K(z,\infty)=\gamma^{\vphantom{p}}_K-V^{\tau^{\vphantom{p}}_K}(z).
\label{green}
\end{equation}

For a real parameter $\theta\geq0$ introduce the mixed Green-logarithmic potential $\theta V ^\mu(z) +G ^\mu_F(z)$, $\mu\in M_1(E)$,  $E=[-1,1]$, $F:=[a,b]$. It is well known \cite{GoRaSu91}, \cite{BuSu15} that there exists a unique probability measure $\lambda_E=\lambda_E(\theta)$ supported on the compact set $E$, $\lambda_E(\theta)\in M_1(E)$, with the following property:
\begin{equation}
\theta V^{\lambda_E}(x)+G^{\lambda_E}_F(x)
\equiv c_E(\theta)=\const,\quad x\in E.
\label{equiv}
\end{equation}
Note that for each $\theta$ we have $\supp{\lambda_E(\theta)}=E$.

Let us introduce another mixed Green-logarithmic potential 
$\theta V ^\nu(z) +G ^\nu_E(z)$, $\nu\in M_1(F)$, and consider  the next equilibrium problem  for that potential with an external field  given by the function $g_E(z,\infty)$:
\begin{equation}
\theta V^{\lambda_F}(y)+G^{\lambda_F}_E(y)
+\theta g_E(y,\infty)\equiv c_F(\theta)=\const,\quad y\in F,
\quad \lambda_F\in M_1(F).
\label{equivconj}
\end{equation}
It is well known  \cite{BuSu15}, \cite{Sue25}  that there exists a unique measure  $\lambda_F=\lambda_F(\theta)$ supported on $F$, $\lambda_F(\theta)\in M_1(F)$, which solves the equilibrium problem \eqref{equivconj}. Note that $\supp{\lambda_F}(\theta)=F$.

There exists (see \cite{BuSu15}, \cite{Sue25}) the following connection\footnote{In fact there is the one-to-one correspondence.} between $\lambda_E(\theta)\in M_1(E)$ and $\lambda_F(\theta)\in M_1(F)$:
\begin{equation}
\lambda_F(\theta)=\beta_F(\lambda_E(\theta)),
 \label{balay}
\end{equation}
where $\beta_F(\cdot)$ is the balayage  of a  given measure outside the domain $G:=\myh {\CC}\setminus{F}$ onto its boundary $\partial G=F$.

There are also some other connections between the potential problems \eqref{equiv} and \eqref{equivconj}, which will be given in Section \ref{s2s3} below.

\subsection{}\label{s1s2}
Let $f\in\mathscr Z(E)$, $f\in\HH(\infty)$, $f(\infty)=1/\sqrt{AB}$, 
\begin{equation}
N=2n+1=3m+1=4\ell+1.
\label{eqN}
\end{equation}

For the pair of functions $f,f ^ 2$  define type II Hermite--Padé polynomials  (see \cite{NiSo88})  $P ^ {(2)}_{2m,0}\not\equiv 0$, $P ^{(2)}_{2m,1}$ and $P^{(2)}_{2m,2}$ of degree $\leq 2m$, $\mdeg{P ^{(2)}_{2m,j}}\leq{2m}$, from the relations:
\begin{equation}
\begin{aligned}
\bigl(P^{(2)}_{2m,0}f-P^{(2)}_{2m,1}\bigr)(z)
&=O(z^{-m-1}),\quad z\to\infty,\\
\bigl(P^{(2)}_{2m,0}f^2-P^{(2)}_{2m,2}\bigr)(z)
&=O(z^{-m-1}),\quad z\to\infty.
\end{aligned}
\label{hepa2}
\end{equation}
Similarly for the triple of functions $f,f ^ 2,f ^ 3$   define  type II Hermite--Padé polynomials $P^{(3)}_{3\ell,0}\not\equiv0,P^{(3)}_{3\ell,1}$, $P^{(3)}_{3\ell,2}$, $P^{(3)}_{3\ell,3}$ of degree $\leq 3\ell$, $\mdeg{P^{(3)}_{3\ell,j}}\leq{3\ell}$, from the following relations:
\begin{equation}
\begin{aligned}
\bigl(P^{(3)}_{3\ell,0}f-P^{(3)}_{3\ell,1}\bigr)(z)
&=O(z^{-\ell-1}),\quad z\to\infty,\\
\bigl(P^{(3)}_{3\ell,0}f^2-P^{(3)}_{3\ell,2}\bigr)(z)
&=O(z^{-\ell-1}),\quad z\to\infty,\\
\bigl(P^{(3)}_{3\ell,0}f^3-P^{(3)}_{3\ell,3}\bigr)(z)
&=O(z^{-\ell-1}),\quad z\to\infty.
\end{aligned}
\label{hepa3}
\end{equation}

Note that under the condition \eqref{eqN} the polynomials $P^{(2)}_{2m,j}$, $j=0,1,2$, and $P^{(3)}_{3\ell,k}$, $k=0,1,2,3$, as well as  Padé polynomials $P^{(1)}_{n,0}$, $P^{(1)}_{n,1}$, 
\begin{equation}
\bigl(P^{(1)}_{n,0}f-P^{(1)}_{n,1}\bigr)(z)
=O\(\frac1{z^{n+1}}\),\quad z\to\infty,
\label{padedef}
\end{equation}
are determined by the first $N$  coefficients $c_0,\dots,c_{N-1}$ of the Laurent expansion of $f\in\mathscr Z(E)$ at the infinity point:
\begin{equation}
f(z)=c_0+\frac{c_1}z+\dots+\frac{c_{N-1}}{z^{N-1}}+\dotsb,
\quad c_0=\frac1{\sqrt{AB}}.
\label{fexp}
\end{equation}
For an arbitrary polynomial $Q\not\equiv0$ let 
$$
\chi(Q):=\sum_{\eta:Q(\eta)=0}\delta_\eta,
$$
where each zero $\eta$  of the polynomial $Q$ is considered by taking into account its multiplicity.
The set of all zeros itself of $Q$ is denoted by $Z(Q)$.

From the work of \cite{ApLy10} follows that as $N\to\infty$
$$
\frac1{2m}\chi(P^{(2)}_{2m,j})\overset{*}\longrightarrow
\lambda_E(3)\quad\text{and}\quad
\frac1{3\ell}\chi(P^{(3)}_{3\ell,k})\overset{*}\longrightarrow
\lambda_E(1),
$$
where ``$\overset{*}\longrightarrow$'' means convergence in the space of measures in weak-$*$ topology, $j=0,1,2$, $k=0,1,2,3$.

The main results of the current paper are the following statements announced for the first time in \cite{DoIkKnSu23} (see also Corollary \ref{cor2}).

\begin{theorem}\label{the1}
Let $f\in\mathscr Z(E)$. Then uniformly inside\footnote{That is on the compact subsets.} of the domain $D=\myh{\CC}\setminus{E}$
the following relation holds as $N\to\infty$: 
\begin{equation}
\lim_{N\to\infty}\frac1N\log\biggl| f(z)-\frac{P^{(2)}_{2m,1}(z)}{P^{(2)}_{2m,0}(z)}\biggr| = -\frac13G^{\lambda_F(3)}_E(z)-g_E(z,\infty)<0.
\label{hp2conv}
\end{equation}
\end{theorem}

\begin{theorem}\label{the2}
Let $f\in\mathscr Z(E)$. Then uniformly inside of the domain $D=\myh{\CC}\setminus{E}$
the following relation holds as $N\to\infty$: 
\begin{equation}
\lim_{N\to\infty}\frac1N\log\biggl|f(z)-\frac{P^{(3)}_{3\ell,1}(z)} {P^{(3)}_{3\ell,0}(z)}\biggr| =-\frac12G^{\lambda_F(1)}_E(z)-g_E(z,\infty)<0.
\label{hp3conv}
\end{equation}
\end{theorem}

Note that in \cite{GoRa81}, \cite{GoRaSo97}, \cite{ApLy10} and \cite{ApLy21} much more  general systems than $[1,f,f^2]$ and $[1,f,f^2,f^3]$ were treated. But the results on limit zero distribution of HP-polynomials  were given there in terms of  vector equilibrium problems. Thus, it is rather difficult to compare the properties of solutions of these vector equilibrium problems corresponding to the pair $f,f^2$ and the triple $f,f^2,f^3$ which systems are of different dimensions. 

From Stahl Theory \cite {Sta97b} it directly follows that inside the domain  $D$
\begin{equation}
\lim_{N\to\infty}\frac1N\log
\bigl|f(z)-[n/n]_f(z)\bigr|=-g_E(z,\infty),
\label{padeconv}
\end{equation}
where $[n/n]_f=P^{(1)}_{n,1}/P^{(1)}_{n,0}$, $N=2n+1$, is the $n$th diagonal Padé approximant to the function $f$ at the infinity point.

The main result of \cite{IkSu24} it the following statement establishing the monotony property of the potential
$G_E^{\lambda_F(\theta)}(z)$ of the equilibrium measure $\lambda_F(\theta)\in M_1(F)$ with respect to the parameter  $\theta$    for $\theta\in[1,3]$.

\begin{theorem}[\cite{IkSu24}, Theorem 2]\label{the3}
Let $1\leq\theta_1<\theta_2\leq3$. Then as $z\in{D}$
\begin{equation}
\(1+\frac1{\theta_1}\)G^{\lambda_F(\theta_1)}_E(z)>
\(1+\frac1{\theta_2}\)G^{\lambda_F(\theta_2)}_E(z).
\label{neqgreeng}
\end{equation}
\end{theorem}

\begin{corollary}\label{cor1}
Set $\theta_1=1$ and $\theta_2=3$. Then from \eqref{neqgreeng} directly follows that 
\begin{equation}
G^{\lambda_F(1)}_E(z)>\frac23 G^{\lambda_F(3)}_E(z),\quad z\in D.
\label{neqgreen13}
\end{equation}
\end{corollary}

\begin{corollary}\label{cor2}
Set
\begin{equation}
\begin{aligned}
\delta_1(z)&:=\exp\bigl\{-g_E(z,\infty)\bigr\},\\
\delta_2(z)&:=\exp\biggl\{-\frac13G^{\lambda_F(3)}_E(z)-g_E(z,\infty)\biggr\}, \\
\delta_3(z)&:=\exp\biggl\{-\frac12G^{\lambda_F(1)}_E(z)-g_E(z,\infty)\biggr\}.
\end{aligned}
\label{deldef}
\end{equation}
Then from \eqref{deldef} and \eqref{neqgreen13} it follows that 
\begin{equation}
0<\delta_3(z)<\delta_2(z)<\delta_1(z)<1,\quad z\in D.
\label{delrel}
\end{equation}
\end{corollary}

\begin{corollary}\label{cor3}
From the proofs of the Theorems \ref{the1} and \ref{the2} it follows that as  $N\to\infty$
$$
\frac1N\chi(P^{(2)}_{2m,0})
\overset{*}\longrightarrow
\lambda_E(3),
\quad 
\frac1N\chi(P^{(3)}_{3\ell,0})
\overset{*}\longrightarrow
\frac13\beta_E(\lambda_F(1))+\frac23\tau_E,
$$
where $\beta_E(\lambda_F(1))$ is the balayage of the measure $\lambda_F(1)$ from $\myh{\CC}\setminus E$ onto $E$,
$\tau_E$ is the probability Chebyshev measure for the set $E$, $d\tau_E=dx/(\pi\sqrt{1-x^2})$, $\supp{\lambda}=E$.
\end{corollary}

Recall (see \cite{GoRaSu91} and cf. \cite{MaMe24}), that the equilibrium measure   $\lambda_E(0)$ and the corresponding function $G ^ {\lambda_E(0)}_F$ are connected with the best (Chebyshev) rational approximations 
of degree $n$ to Markov type function $\myh{\varkappa}(z)$ on the segment $E=[-1,1]$, $\supp{\varkappa}=F=[a,b]$. However these approximants are not constructive   ones in the sense of the paper by P. Henrici \cite[Sec. 2]{Hen66}. For more discussion about constructive approximation see also \cite{Ara84}, \cite{Sta06}, \cite{Tre23}, \cite{AnAv24} and the bibliography therein. Remark also that very often the polynomials $P^{(2)}_{2m,0}$ and $P^{(3)}_{3\ell,0}$ are called multiple orthogonal polynomials and  are treated without connection with constructive rational approximations, see \cite{ApLy10}, \cite{ApLy21}, \cite{Lys20}, \cite{Lys24}, \cite{ApNo25}, \cite{Apt26}.

Now suppose that (see equation \eqref{eqN}):
\begin{equation}
N=2n+1=3m+2=4\ell+3.
\label{qeNN}
\end{equation}
Then based on the $N$ Laurent coefficients $c_0,\dots,c_{N-1}$ of $f$ at the infinity point it is possible to construct (see \cite{NiSo88}, \cite{IkSu21}) type I Hermite--Padé polynomials  $Q^{(2)}_{m,j}$, $j=0,1,2$ of degree $m$ for the tuple $[1,f,f^2]$, as well as type I Hermite--Padé polynomials  $Q^{(3)}_{\ell,k}$, $k=0,1,2,3$ of degree $\ell$ for the tuple $[1,f,f^2,f^3]$:
\begin{align}
(Q^{(2)}_{m,0}+Q^{(2)}_{m,1}f+Q^{(2)}_{m,2}f^2)(z)&=O\(\frac1{z^{2m+2}}\),&& z\to\infty,\label{eqhp2}\\
(Q^{(3)}_{\ell,0}+Q^{(3)}_{\ell,1}f+Q^{(3)}_{\ell,2}f^2+Q^{(3)}_{\ell,3}f^3)(z)&=O\(\frac1{z^{3\ell+3}}\),&& z\to\infty.
\label{eqhp3}
\end{align}
Based on the {\it rational} Hermite--Padé approximants, corresponding to $Q^{(2)}_{m,j}$ and $Q^{(3)}_{\ell,k}$, it is possible (see \cite{Sue18d}, \cite{Kom21}) to recover the values $f(z^{(1)})$ and $f(z^{(2)})$ of a multivalued function $f\in\mathscr Z(E)$, $f(z^{(0)}):=f(z)$, $z\in D$, on the other sheets of its four sheeted Riemann surface (see \cite{KoPaSuCh17}, \cite{Sue18d}, \cite{IkSu21}, \cite{Kom21}, \cite{Sue23}, \cite{Sue24}). It is also possible \cite{IkSuHePa21}, \cite{KoPa24} to recover the branch points of $f$ with the discriminants corresponding to \eqref{eqhp2} and \eqref{eqhp3}:
\begin{align}
D_m^{(2)}(z)&:=(Q^{(2)}_{m,1}(z))^2-4Q^{(2)}_{m,0}(z)Q^{(2)}_{m,2}(z),\label{eqdis2}\\
D_\ell^{(3)}(z)&:=\Bigl(18Q^{(3)}_{\ell,3}Q^{(3)}_{\ell,2}Q^{(3)}_{\ell,1}Q^{(3)}_{\ell,0}-4(Q^{(3)}_{\ell,2})^3Q^{(3)}_{\ell,0}\nonumber\\
&\phantom{:}+(Q^{(3)}_{\ell,2})^2(Q^{(3)}_{\ell,1})^2-4Q^{(3)}_{\ell,3}(Q^{(3)}_{\ell,1})^3-27(Q^{(3)}_{\ell,3})^2(Q^{(3)}_{\ell,0})^2\Bigr)(z).\label{eqdis3}
\end{align}
We propose to discuss these results concerning the type I Hermite--Padé polynomials in the next paper.

Finally, note that starting from the seminal paper by A. A. Gonchar and E. A. Rakhmanov \cite{GoRa81} all the problems concerning the so-called weak asymptotics of Hermite--Padé polynomials for Markov type functions were treated by the vector potential-theoretic method (see \cite{GoRaSo97}, \cite{ApLy10}, \cite{Lys20}, \cite{ApLy21}, \cite{Lys24}).  In contrast to that method, in the current paper we use an approach based on the scalar equilibrium problems \eqref{equiv} and \eqref{equivconj}.

\begin{remark}\label{rem11}
If we consider the complex case, i.e. the class $\CC(z,f)$  of analytic functions for a fixed function $f\in\mathscr Z(E)$ defined  by \eqref{funzhu} then  clearly Theorems~\ref{the1} and \ref{the2} are still valid. However the proofs become more complicated, since there are no interpolation properties, which take place in the real case situation (see, for example, equations \eqref{pqsig} and \eqref{repROm}).
\end{remark}

\begin{remark}\label{rem12}
The situation becomes much more complicated  if we change the geometrical  component of the problem under consideration, i.e. admit the points $a$ and $b$ to be complex but not real. In  general, in this situation for the pair $f,f^2$ and the triple $f,f^2,f^3$ instead of the segments $E=[-1,1]$ and $F=[a,b[$ there appear two different Nuttall's condensers (see \cite{RaSu13}) each of which consists of two  pairs of $S$-curves (see \cite{Rak12}).  As usual, to describe explicitly these $S$-properties it is necessary to apply variational methods (see \cite{Rak12}, \cite{Dub24}). However, the corresponding domains of convergence became different from each other and it is impossible to compare the rates of convergence in the sense of Theorems~\ref{the1} and \ref{the2}.
\end{remark}

\section{Some Auxiliary Results}\label{s2}

\subsection{}\label{s2s1}
In the current section we state some auxiliary results which we need for the proofs of Theorems~\ref{the1} and  \ref{the2}. Note that in general these results are well-known, see \cite{GoRa87}, \cite{BuSu15}, \cite{Rak16}. However for the purpose of the current paper we need these results in specific forms.

Let $\mu\in M_1(E)$ and 
\begin{equation}
J(\mu):=\iint_{E\times E}
\log\frac1{|x-y|}\,d\mu(x)d\mu(y)
=\int_E V^\mu(x)\,d\mu(x)
\label{enelog}
\end{equation}
be the logarithmic energy of a measure $\mu$.
Let  $M^\circ_1(E)$ be the set of all measures from $M_1(E)$ such that $J(\mu)<\infty$. 
Recall that when $\mu,\nu\in M^\circ_1(E)$ then the corresponding value $[\mu,\nu]$ is finite (see \cite[Chapter II, Section 5, p. 127, Lemma 5.4]{SaTo97} and Remark \ref{rem22} below):
\begin{equation}
[\mu,\nu]:=\iint_{E\times E}\log\frac1{|x-y|}\,d\mu(x)d\nu(y)
=[\nu,\mu]<\infty,
\label{enemut}
\end{equation}
and we have also that $J(\mu-\nu)\geq0$ with equality for the case when $\mu=\nu$ only (the positiveness of the logarithmic kernel, see  \cite{Lan66}, \cite{SaTo97}, \cite{Chi18}, \cite{Chi19}, \cite{Chi20}).

The following result follows directly from \cite[\S 3, Lemma 6]{GoRa87}.

\begin{lemma}\label{lem1} Let $\psi(x)$ be a continuous function given on $E$. Then there exists a unique probability measure $\lambda=\lambda_{\equ}=\lambda(\psi)$ supported on $E$ with the following property:
\begin{equation}
V^{\lambda}(x)+\psi(x)
\begin{cases}
\leq c(\psi),& x\in \supp{\lambda},\\
\geq c(\psi), & x\in E\setminus\supp{\lambda},
\end{cases}
\label{equpro}
\end{equation}
with some constant $c(\psi)=\const$.
\end{lemma}

\begin{corollary}\label{cor4}
Let $\psi(x)=-V^\mu(x)$, where $\supp(\mu)\Subset D=\RR\setminus{E}$.
Then $\lambda(\psi)=\beta_E(\mu)$ is the balayage of $\mu$ from $D$ onto $E$,  $\supp{\lambda}=E$ and 
\begin{equation}
V^{\lambda(\psi)}(z)\equiv
V^\mu(z)-G_E^\mu(z)+c(\psi),\quad
c(\psi)=\int g_E(y,\infty)\,d\mu(y).
\label{cpsi}
\end{equation}
\end{corollary}

\begin{corollary}\label{cor5}
Let $\psi(x)=-tV^\mu(x)$, $t\in(0,1)$, where $\supp(\mu)\Subset D=\RR\setminus{E}$. Then $\lambda(\psi)=(1-t)\tau_E+t\beta_E(\mu)$, where $\beta_E(\mu)$ is the balayage of $\mu$ from $D$ onto $E$,  $\tau_E$ is the probability  Chebyshev measure for the set $E$, $d\tau_E=dx/(\pi\sqrt{1-x^2})$, $\supp{\lambda}=E$, and 
\begin{align}
V^{\lambda(\psi)}(z)&\equiv (1-t)V^{\tau_E}(z)+t(V^\mu(z)-G_E^\mu(z))+c(\psi)\nonumber\\
&= (1-t)\gamma_E-(1-t)g_E(z,\infty)+t(V^\mu(z)-G_E^\mu(z))+c(\psi),\label{psitv}\\
c(\psi)&=(1-t)\gamma_E+t\int g_E(y,\infty)\,d\mu(y),\quad\gamma_E=\log2.\nonumber
\end{align}
\end{corollary}

\begin{remark}\label{rem21}
In general we have that $V^\lambda(x)+\psi(x)=c(\psi)$ quasi-everywhere (qua.e.) on $\supp{\lambda}$ in equation \eqref{equpro}, see \cite[Chapter 2, \S 4]{Lan66}, \cite[Chapter 5, Section 3]{NiSo88}, \cite[Chapter I, Section I.1]{SaTo97} and  the end of the proof of Lemma \ref{lem1} below.
Note that
\begin{equation}
c(\psi)=\int_E(V^\lambda+\psi)(x)\,d\lambda(x).
\label{cpsifo}
\end{equation}
Also from \cite[Theorem 4.1]{NiSo88} it follows that under the conditions of Lemma \ref{lem1} we have that $V^{\lambda}(x)+\psi(x)\equiv c(\psi)$ for $x\in\supp{\lambda}$.
\end{remark}

\begin{remark}\label{rem22}
Let $\Omega\Subset\CC$ be a domain which possesses a Green function $g_\Omega(z,\zeta)$,
and $\mu$ and $\nu$ be two positive measures on $\Omega$, $\supp{\mu},\supp{\nu}\subset \Omega$. Let 
\begin{equation}
I_\Omega(\mu):=\iint g_\Omega(z,\zeta)\,d\mu(z)\,d\mu(\zeta)\quad
\text{and}\quad 
I_\Omega(\nu):=\iint g_\Omega(z,\zeta)\,d\nu(z)\,d\nu(\zeta)
\label{greenin}
\end{equation}
be the Green's energies of the measure $\mu$ and $\nu$ respectevely, and
\begin{equation}
I_G(\mu,\nu):=\iint  g_\Omega(z,\zeta)\,d\mu(z)\,d\nu(\zeta)
\label{greenmu}
\end{equation}
be the mutual Green energy of the measures $\mu$ and $\nu$. 

The proofs in \cite[Chapter II, Section 5, Lemma 5.4]{SaTo97} show that if the measures $\mu$ and $\nu$ are of finite Green energy each, then (i) $I^2_\Omega(\mu,\nu)\leq I_\Omega(\mu)I_\Omega(\nu)<\infty$  and (ii) the measure $\mu+\nu$ is also of finite Green energy.

Since $g_\Omega(z,\zeta)=-\log|z-\zeta|+h_\Omega(z,\zeta)$, where the function $h_\Omega(z,\zeta)$ is  a continuous function for $z,\zeta\in K$, $K\Subset\Omega$, then from the above result for the Green energy it follows a similar result for the logarithmic energy \eqref{enelog}. More precisely, let the measures $\mu$ and $\nu$, $\supp{\mu},\supp{\nu}\Subset\CC$,  be of finite logarithmic energies. Then the measure $\mu+\nu$ is also of finite logarithmic energy and their mutual energy $[\mu,\nu]$ is finite, $[\mu,\nu]<\infty$.

Clearly, the above  conclusions are valid for the energy functional $J_\psi(\cdot)$ \eqref{enefil} of a measure in the presence of a continuous external field $\psi$.
\end{remark}

The measure $\lambda$ from \eqref{equpro} is called the {\it equilibrium measure} in the external field $\psi$ and the constant $c(\psi)=\const$ is called the {\it equilibrium constant} (see  \cite{GoRa87}, \cite{SaTo97}). Lemma \ref{lem1} follows directly from the \cite[\S 3, Lemma 6]{GoRa87}, \cite{NiSo88} but we give here its proof for completeness.

\begin{proof}[Proof of Lemma \ref{lem1}]
Let $\mu\in M_1(E)$ and
\begin{align}
J_\psi(\mu):&=\iint_{E\times E}\biggl(
\log\frac1{|x-y|}+\psi(x)+\psi(y)\biggr)\,d\mu(x)d\mu(y)\nonumber\\
&=\int_E V^\mu(x)\,d\mu(x)+2\int_E \psi(x)\,d\mu(x)
=J(\mu)+2\int_E \psi(x)\,d\mu(x)
\label{enefil}
\end{align}
be the energy functional of the measure $\mu$ in the external field $\psi(x)$ ($\psi$-weighted energy of $\mu$). Recall that $M^\circ_1(E)$ is the set all measures from $M_1(E)$ such that $J(\mu)<\infty$, and hence $J_\psi(\mu)<\infty$ for  all $\mu\in M_1^\circ(E)$.

Set
\begin{equation}
m:=\inf_{\mu\in M^\circ_1(E)}J_\psi(\mu)<\infty.
\label{minpro}
\end{equation}
Then there exists a sequence $\{\mu_n\}$, $\mu_n\in M^\circ_1(E)$, such that $J_\psi(\mu_n)=m_n\to m$ as $n\to\infty$. From the weak compactness of the space $M_1(E)$ it follows that $\mu_n\overset{*}\rightarrow\lambda\in M_1(E)$ as $n\to\infty$ for $n\in\Lambda$,
where $\Lambda\subset\NN$ is an infinite subsequence of $\NN$. 
Then according to the descendence principle (see \cite[Chapter I, \S 3, Theorem 1.3]{Lan66}, \cite{SaTo97}, \cite{Chi18} and \cite{Chi20}) we obtain that
$$
\varliminf_{\substack{n\to\infty\\n\in\Lambda}}
J_\psi(\mu_n)\geq J_\psi(\lambda).
$$
Therefore $J_\psi(\lambda)=m$, $\lambda\in M_1^\circ(E)$ and $\lambda=\lambda_{\min}$ is the so-called {\it minimizer} for \eqref{minpro}. Now let us prove that for the minimizer $\lambda$ the equilibrium property \eqref{equpro} is fulfilled, i.e. $\lambda_{\min}=\lambda_{\equ}=\lambda(\psi)$. 

Let $\mu,\nu\in M_1^\circ(E)$. It is easy to see that
\begin{align}
2J_\psi(\mu)+2J_\psi(\nu)-4J_\psi\biggl(\frac{\mu+\nu}2\biggr)
&=2J(\mu)+2J(\nu)-4J\biggl(\frac{\mu+\nu}2\biggr)\label{sumine}\\
&=J(\mu)+J(\nu)-[\mu,\nu]-[\nu,\mu]=J(\mu-\nu).\nonumber
\end{align}
Since $J(\mu-\nu)\geq0$ (see \cite{Chi19}) then from \eqref{sumine} we obtain that
\begin{equation}
2J_\psi\biggl(\frac{\mu+\nu}2\biggr)\leq J_\psi(\mu)
+J_\psi(\nu).
\label{sumneq}
\end{equation}
It is easy to see that for $t\in[0,1]$ and any $\nu$ we have
\begin{equation}
J_\psi(t\nu+(1-t)\lambda)-J_\psi(\lambda)
=2t\int_E(V^\lambda+\psi)(x)\,d(\nu-\lambda)(x)
+t^2J(\nu-\lambda).
\label{perine}
\end{equation}
Indeed,
\begin{align}
J_\psi(t\nu&+(1-t)\lambda)=J_\psi(t(\nu-\lambda)+\lambda)\nonumber\\
&=\int_EV^{t(\nu-\lambda)+\lambda}(x)\,d(t(\nu-\lambda)+\lambda)(x)
+2\int_E\psi(x)\,d(t(\nu-\lambda)+\lambda)(x)\nonumber
\end{align}
\begin{align}
&=\int_EV^{t(\nu-\lambda)}(x)\,d(t(\nu-\lambda))(x)
+\int_E V^{t(\nu-\lambda)}(x)\,d\lambda(x)\nonumber\\
&\quad
+\int_EV^\lambda(x)\,d(t(\nu-\lambda))(x)
+\int_E V^\lambda(x)\,d\lambda(x)\nonumber\\
&\quad+2\int_E\psi(x)\,d(t(\nu-\lambda))(x)+2\int_E\psi(x)\,d\lambda(x)
\nonumber\\
&=J_\psi(\lambda)+2\int_E(V^\lambda+\psi)(x)\,d(t(\nu-\lambda))(x)
+t^2J(\nu-\lambda)\nonumber\\
&=J_\psi(\lambda)+2t\int(V^\lambda+\psi)(x)\,d(\nu-\lambda)(x)
+t^2J(\nu-\lambda).
\label{perine2}
\end{align}
From \eqref{perine2} the equality \eqref{perine} follows.

Since $J(\nu-\mu)\geq0$, then from \eqref{perine} it follows that the minimizer $\lambda$ is the only measure which satisfies the inequality 
\begin{equation}
\int(V^\lambda+\psi)(x)\,d(\nu-\lambda)(x)\geq0
\label{ineqnu}
\end{equation}
for any $\nu\in M_1^\circ(E)$. 

Indeed, if \eqref{ineqnu} is satisfied by any $\nu$, then for $t=1$ we obtain from \eqref{perine} that $J_\psi(\nu)\geq J_\psi(\lambda)$ since $J(\nu-\mu)\geq0$. Thus $\lambda$ is a minimizer.

Conversely, if $\lambda=\lambda_{\min}$ is a minimizer, then $0\leq J_\psi((t\nu+(1-t)\lambda)-J_\psi(\lambda)$, $t\in(0,1)$. Since $J(\nu-\mu)\geq0$ then from the right-side of \eqref{perine}  by sending $t\to0$ we  obtain \eqref{ineqnu}.

The uniqueness of the minimizer follows from \eqref{sumine}, the inequality $J(\mu-\nu)\geq0$ and the relation $J(\mu-\nu)=0$ $\Leftrightarrow$ $\mu=\nu$ (see \cite{Chi18}, \cite{Chi19}).

Let us prove now that the minimizer $\lambda=\lambda_{\min}$ is an equilibrium measure \eqref{equpro} with
$$
c'(\psi):=\int_E (V^\lambda+\psi)(x)\,d\lambda(x).
$$
Suppose that $(V^\lambda+\psi)(x)<c'(\psi)$ on the set  $e\subset E$ of positive (inner) capacity, $\mcap{e}>0$. Then there exists \cite{Anc83} a regular compact set $e_1\subset e$ of positive capacity.  Let $\tau_{e_1}$ be the (probability) Robin measure for $e_1$. Then we have
$$
\int_E(V^\lambda+\psi)(x)\,d(\tau_{e_1}-\lambda)(x)
=\int_E(V^\lambda+\psi)(x)\,d\tau_{e_1}(x) -c'(\psi)<0.
$$
A contradiction with \eqref{ineqnu}. 
Therefore  $(V^\lambda+\psi)(x)\geq c'(\psi)$ qua.e. on $E$. 

The potential $V^\lambda$ is semi-continuous from below. Hence if $(V^\lambda+\psi)(x_0)>c'(\psi)$ on the $\supp{\lambda}$, $x_0\in\supp{\lambda}$,  then $(V^\lambda+\psi)(x)>c'(\psi)$ for $x\in U_\delta(x_0):=\{x\in E:|x-x_0|<\delta\}$, $\lambda(U_\delta)>0$, for some $\delta>0$. Therefore\footnote{Note that since $\lambda$ is of finite energy, then the integration with $\lambda$ does not depend on the set of zero capacity where $(V^\lambda+\psi)(x)<c'(\psi)$.}
$$
\int_E(V^\lambda+\psi)(x)\,d\lambda(x)>c'(\psi).
$$
A contradiction with the definition of $c'(\psi)$.

If $\lambda=\lambda_{\equ}$ is an equilibrium measure for \eqref{equpro} then $(V^\lambda+\psi)(x)\leq c'(\psi)$ on $\supp{\lambda}$. Therefore $\lambda\in M_1^\circ(E)$. We have also that for any $\nu\in M_1^\circ(E)$
\begin{align*}
\int_E(V^\lambda+\psi)(x)\,d(\nu-\lambda)(x)&=\int_E(V^\lambda+\psi)(x)\,d\nu(x)\\
&-\int_E(V^\lambda+\psi)(x)\,d\lambda(x)\geq c'(\psi)-c'(\psi)=0.
\end{align*}
Hence by \eqref{ineqnu} any equilibrium measure $\lambda_{\equ}$ coincides with the unique minimizer $\lambda_{\min}$ and  thus it is also unique. Clearly \eqref{cpsifo} is also valid.
\end{proof}

\subsection{}\label{s2s2}
The next result is the cornerstone in treating the asymptotic properties of Hermite--Padé polynomials (see \cite{GoRa81}, \cite{GoRa87}, \cite{RaSu13}, \cite{Rak16}, \cite{Sue18}, \cite{Sue25b}, \cite{RaSu25}).

\begin{lemma}\label{lem2}
Let $\{\Psi_n\}$ be a sequence of functions which are continuous on $E$ and positive on $E$, $\Psi_n(x)>0$, $x\in E$, $n=0,1,2,\dots$, and such that  as $n\to\infty$
$$
\frac1{2n}\log\Psi_n(x)\rightrightarrows \psi(x),\quad x\in E,
$$
where $\psi(x)$ is a continuous function on  $E$.
Suppose also that $\sigma$ is a positive Borel measure supported on $E$,
and such that $\sigma'(x)=d\sigma(x)/dx>0$ almost everywhere on $E$. 

Let $Q_n(x)=Q_n(x,\Psi_n)=x^n+\dotsb$ be polynomials that are orthogonal with respect to the variable weight function on $E$:
\begin{equation}
\int_{-1}^1 Q_n(x)x^k\Psi_n(x)\,d\sigma(x)=0,
\quad k=0,\dots,n-1.
\label{ortvar}
\end{equation}
Then as $n\to\infty$
\begin{equation}
\frac1n\chi(Q_n)\overset{*}\longrightarrow\lambda(\psi),
\label{convar}
\end{equation}
where $\lambda(\psi)=\lambda_{\equ}$ is the equilibrium measure from \eqref{equpro}.
\end{lemma}

\begin{remark}\label{rem23}
Note that the orthogonality relations \eqref{ortvar} are equivalent to the following ones:
\begin{equation}
\int_{-1}^1 Q_n(x)P_{n}(x)\Psi_n(x)\,d\sigma(x)=0
\label{ortvar2}
\end{equation}
with any $P_n\in\PP_{n-1}$, i.e., any polynomial $P_n$ of $\mdeg{P_n}\leq{n-1}$. Therefore, the external field $\Psi_n$ in the relation \eqref{ortvar2} can depend on the polynomial $Q_n$ itself, see \cite[formula (119)]{RaSu13}, \cite[formula (41)]{Rak16}, \cite[Section 2.5]{Sue24} and cf. \eqref{ortuQ} below. 
Indeed, in \cite[formula (41)]{Rak16} we can set $P_n(x)=Q_n(x)/(x-x_n)$, $Q_n(x_n)=0$, to obtain the external field that depends on $Q_{n}$ itself. 

Moreover, since $\mdeg{Q_n}=n$, all zeros of $Q_n$ are simple and belong to the open interval $(-1,1)$, then  the orthogonality relations \eqref{ortvar2} are equivalent to the orthogonal relations with polynomials of type $Q_n(x)/((x-\zeta_{n,1})(x-\zeta_{n,2}))$ of degree $n-2$, where the points $\zeta_{n,1}\neq\zeta_{n,2}$ vary throughout the set of all zeros of $Q_n$. And again, the external field $\Psi_n$ can depend on the polynomial $Q_n$ through the polynomial $Q_n(x)/((x-\zeta_{n,1})(x-\zeta_{n,2}))$.

Also, instead of the case $n=0,1,2,\dots$ it is possible to consider the case when $n\in\Lambda$, where $\Lambda\subset\NN$ is an infinity subsequence of $\NN$. 
\end{remark}

\begin{remark}\label{rem24}
Note that from the proof of Lemma \ref{lem2} it directly follows that uniformly for $y\in K\Subset\RR\setminus{E}$
\begin{align}
\lim_{n\to\infty}
\biggl|
\int_{-1}^1\frac{Q_n^2(x)\Psi_n(x)}{x-y}\,d\sigma(x)
\biggr|^{1/n}&=e^{-2c(\psi)},
\label{firstcoefy}\\
\lim_{n\to\infty}\biggl(
\int_{-1}^1 Q_n^2(x)\Psi_n(x)\,d\sigma(x)
\biggr)^{1/n}&=e^{-2c(\psi)},
\label{firstcoef}
\end{align}
\end{remark}

\begin{remark}\label{rem25}
Since  for each $Q\in\PP_n$, $Q\not\equiv0$,
$(Q(z)-Q(x))/(z-x)$ is a polynomial of degree $\leq{n-1}$, then
$$
\int Q_n(x)\frac{Q(z)-Q(x)}{z-x}\Psi_n(x)
\,d\sigma(x)=0.
$$
Thus
\begin{align}
\int \frac{Q_n(x)}{z-x}\Psi_n(x)\,d\sigma(x)
&=\frac1{Q(z)}\int \frac{Q_n(x)Q(x)}{z-x}\Psi_n(x)\,d\sigma(x)\nonumber\\
&=\frac1{Q_n(z)}\int \frac{Q^2_n(x)}{z-x}\Psi_n(x)\,d\sigma(x).\label{relQint}
\end{align}
\end{remark}

\begin{proof}[Proof of Lemma \ref{lem2}]
The proof of Lemma \ref{lem2} is based on the Gonchar--Rakhmanov--Stahl method, created in 1985--1987
(see \cite{Sta97b}, \cite{GoRa87}, \cite{RaSu13}, \cite{Rak16} and the bibliography therein).

As usual when using\footnote{Note that under the conditions of Lemma \ref{lem2} the application of the $\GRS$-method is much more easier than in general situation since we are in a real case situation, see \cite{Sta97b}, \cite{GoRa87}, \cite{Rak16}.} Gonchar--Rakhmanov--Stahl method (in short $\GRS$-method) we shall apply the method from the opposing side. Thus, suppose  that as  $n\to\infty$
\begin{equation}
\frac1n\chi(Q_{n})\not\rightarrow\lambda=\lambda(\psi).
\label{qnotla}
\end{equation}
Then from \eqref{qnotla} and based on orthogonality relations \eqref{ortvar} we will come to a contradiction.

Recall that we are in a real case situation and thus from the orthogonality conditions \eqref{ortvar} it follows that $\mdeg{Q_n}=n$ and all its zeros $x_{n,1},\dots,x_{n,n}$ are real, simple  and belong to the open interval $(-1,1)$. Therefore, from the weak compactness of the space of probability measures $M_1(E)$ it follows that for an infinite subsequence $\Lambda\subset\NN$ we have
\begin{equation}
\frac1n\chi(Q_{n}){\overset{*}\longrightarrow}\mu\neq\lambda,
\quad n\in\Lambda,\quad n\to\infty,
\label{qlimmu}
\end{equation}
$\supp{\mu}\subset{E}$. 
We show now that the above relation \eqref{qlimmu} and the orthogonality conditions \eqref{ortvar} are in contradiction with each other.

Indeed, since $\mu\neq\lambda$, then
for $x\in \supp{\mu}\subset E$ we have that
$$
V^\mu(x)+\psi(x)\not\equiv m_0
:=\min_{z\in E}\bigl(V^\mu(x)+\psi(x)\bigr)=V^\mu(x_0)+\psi(x_0),
$$
where $x_0\in E$. Thus, there exists a point
$x_1\in \supp{\mu}\subset E$, $x_1\neq x_0$, and $\eps_0>0$ such that 
\begin{equation}
V^\mu(x_1)+\psi(x_1)=m_1>m_0+\eps_0.
\label{neqx1}
\end{equation}
Since $\psi(x)$ is a continuous function on $E$ and the potential
$V^\mu(x)$ is a semi-continuous function from below then quite the same inequality \eqref{neqx1} is valid in a $\delta$-neighbourhood 
$U_\delta(x_1):=(x_1-\delta,x_1+\delta)\not\ni x_0$, $\delta>0$, of the point $x_1$. Since $x_1\in \supp{\mu}$, then $\mu(U_\delta(x_1))\geq\delta_0>0$. Therefore, for $n$ large enough, $n\geq n_0$, $n\in\Lambda$, there exists a polynomial $p_n(x)=(x-\zeta_{n,1})(x-\zeta_{n,2})$ such that $\zeta_{n,1},\zeta_{n,2}\in U_\delta(x_1)$ and $p_n(x)$  divides the polynomial $Q_{n}$, i.e., $Q_{n}(x)/p_n(x)\in\PP_{n-2}$.

Set $x_{n-1,n}=\zeta_{n,1}$, $x_{n,n}=\zeta_{n,2}$ and
\begin{equation}
\myt{Q}_n(x):=\frac{Q_{n}(x)}{p_n(x)}=\prod_{j=1}^{n-2}(z-x_{n,j}).
\label{qtilde}
\end{equation}
Now from the orthogonality relations \eqref{ortvar} we obtain that
\begin{equation}
0=\int_{E\setminus{U_\delta(x_1)}}
\frac{Q_{n}^2(x)}{p_n(x)}\Psi_n(x)\,d\sigma(x)
+\int_{\myo{U}_\delta(x_1)}
\frac{Q_{n}^2(x)}{p_n(x)}\Psi_n(x)\,d\sigma(x).
\label{orttil}
\end{equation}
Let us denote the first integral in \eqref{orttil} with $I_{n,1}$ and the second one with $I_{n,2}$. Since $\mdeg{p_n}=2$, then for $x\in E\setminus{U_\delta(x_1)}$ the sign of  the integrand in $I_{n,1}$  is constant. Therefore, we have that 
\begin{equation}
|I_{n,1}|=\int_{E\setminus{U_\delta(x_1)}}\biggl|
\frac{Q_{n}^2(x)}{p_n(x)}\biggr|\Psi_n(x)\,d\sigma(x).
\label{ortint1}
\end{equation}
From the known methods of the theory of logarithmic potential, and similarly to \cite[Lemma 7]{GoRa87}, there exists a limit 
\begin{equation}
\lim_{\substack{n\to\infty\\n\in\Lambda}} |I_{n,1}|^{1/2n}
=\exp\biggl\{
-\min_{x\in E\setminus{U_\delta(x_1)}}\bigl(V^\mu(x)+\psi(x)
\bigr)\biggr\}
=e^{-m_0}.
\label{limint1}
\end{equation}
Moreover, the proof of the relation \eqref{limint1} is quite similar  to the proof of \cite[Lemma 7]{GoRa87}. We have that as $n\to\infty$
\begin{equation}
\min_{x\in E}\biggl\{-\frac1{2n}
\log\bigl(
|Q_{n}(x)\myt{Q}_n(x)|\Psi_n(x)\bigr)\biggr\}
\to\min_{x\in E} \bigl\{V^\mu(x)+\psi(x)\bigr\}.
\label{limlogq}
\end{equation}
From \eqref{limlogq} follows the relation
\begin{equation}
\max_{x\in E}\bigl\{|Q_{n}(x)\myt{Q}_n(x)|\Psi_n(x)
\bigr\}^{1/2n}\to\exp\bigl\{-\min_{x\in E}\bigl[V^\mu(x)+\psi(x)\bigr]\bigr\}=e^{-m_0}
\label{limqpsi}
\end{equation}
as $n\to\infty$, $n\in\Lambda$, and hence the upper estimate is
\begin{equation}
\varlimsup_{\substack{n\to\infty\\n\in\Lambda}} |I_{n,1}|^{1/n}
\leq e^{-m_0}.
\label{upesti1}
\end{equation}

Now let us prove the corresponding estimate from below.
The potential $V^\mu$ is continuous in the thin topology \cite[Chapter 5, \S 3]{Lan66}.
Hence, the function $V^\mu+\psi$ is approximatively continuous with respect
to Lebesgue measure on the compact $E$. From there it follows that  for any $\eps>0$ the set
$$
e=\{x\in E: (V^\mu+\psi)(x)<m_0+\eps\}
$$
is of positive Lebesgue measure, $|e|>0$. From the above we obtain the convergence in Lebesgue measure on $E$ as $n\to\infty$
$$
-\frac1{2n}\log
\bigl\{|Q_{n}(x)\myt{Q}_n(x)|\Psi_n(x)\bigr\}
\overset{\meas}\longrightarrow
(V^\mu+\psi)(x).
$$
Thus, the Lebesgue measure of the set
$$
e_n:=\biggl\{x\in e:-\frac1{2n}\log
\bigl(|Q_{n}(x)\myt{Q}_n(x)|\Psi_n(x)\bigr)
<m_0+\eps\biggr\}
$$
converges to the measure of the set $e$ as $n\to\infty$, $n\in\Lambda$. Hence,
\begin{align}
\varliminf_{\substack{n\to\infty\\n\in\Lambda}} |I_{n,1}|^{1/2n}
&\geq e^{-(m_0+\eps)}\lim_{\substack{n\to\infty\\n\in\Lambda}}
\(\int_{e_n}\,d\sigma(x)\)^{1/2n}\nonumber\\
&= e^{-(m_0+\eps)}\lim_{\substack{n\to\infty\\n\in\Lambda}}
\(\sigma(e_n)\)^{1/2n}=e^{-(m_0+\eps)}.
\label{lestint1}
\end{align}
The last inequality in \eqref{lestint1} is valid since $|e|>0$ and $\sigma'(x)>0$ a.e. on $E$.
Since $\eps>0$ is an arbitrary number, then from \eqref{lestint1} follows that the lower estimate is
$$
\varliminf_{\substack{n\to\infty\\n\in\Lambda}} |I_{n,1}|^{1/n}
\geq e^{-m_0}.
$$
The relation \eqref{limint1} has been proved. 

Let us return to \eqref{orttil}, for the second integral $I_{n,2}$ the following estimate from above is valid
\begin{equation}
|I_{n,2}|\leq
\int_{\myo{U}_\delta(x_1)}
|Q_n(x)\myt{Q}_n(x)|\Psi_n(x)\,d\sigma(x).
\label{uestint2}
\end{equation} 
Based on relations \eqref{uestint2} and \eqref{neqx1}, and using arguments quite similar to the previous, we obtain that
\begin{equation}
\varlimsup_{\substack{n\to\infty\\n\in\Lambda}} |I_{n,2}|^{1/n}
\leq\exp\biggl\{-\min_{x\in \myo{U}_\delta(x_1)}
\bigl(V^\mu(x)+\psi(x)\bigr)
\biggr\}<e^{-(m_0+\eps_0)}<e^{-m_0}.
\label{ulimint2}
\end{equation}
The relations \eqref{limint1} and \eqref{ulimint2} are in a contradiction with each other since in view of orthogonality relation \eqref{orttil} we have that $I_{n,1}=-I_{n,2}$. 

The proof of Lemma \ref{lem2} is complete.
\end{proof}

Recall that if $\mu,\nu\in M^\circ_1(E)$ then the value $[\mu,\nu]$ is finite (see \cite{SaTo97}, \cite{Chi18}, \cite{Chi19}, \cite{Chi20} and Remark \ref{rem22}):
$$
[\mu,\nu]:=\iint_{E\times E}\log\frac1{|x-y|}\,d\mu(x)d\nu(y)
=[\nu,\mu]<\infty,
$$
and we also have that $J(\mu-\nu)\geq0$ with equal sign only for the case when $\mu=\nu$ (the positiveness of the logarithmic kernel, \cite{Lan66}, \cite{SaTo97}, \cite{Chi18}, \cite{Chi19}, \cite{Chi20}).

\subsection{}\label{s2s3}
Note that there are some useful connections between the solutions of the potential equilibrium
problems \eqref{equiv} and \eqref{equivconj} (see \cite[formula (24)]{BuSu15}, \cite{IkSu24}, \cite{Sue25})
\begin{align}
\theta V^{\lambda_F(\theta)}(z)+G^{\lambda_F(\theta)}_E(z)
+\theta g_E(z,\infty)+(1+\theta)G^{\lambda_E(\theta)}_F(z)&\equiv c_F(\theta),\quad z\in\myh{\CC},\label{connect1}\\
\theta V^{\lambda_E(\theta)}(z)+G^{\lambda_E(\theta)}_F(z)+\theta g_E(z,\infty)+G^{\lambda_F(\theta)}_E(z)&\equiv c_E(\theta),\quad z\in\myh{\CC},\label{connect2}
\end{align}
where $\lambda_F(\theta)=\beta_F(\lambda_E(\theta))$.
The relation \eqref{connect1} has been proved in \cite[formula (24)]{BuSu15} (see also \cite{IkSu24} and \cite{Sue25}).  The relation \eqref{connect2} follows from the fact that the function
$$
v(z):=\theta V^{\lambda_E(\theta)}(z)+G^{\lambda_E(\theta)}_F(z)+\theta g_E(z,\infty)+G^{\lambda_F(\theta)}_E(z)
$$
is a harmonic function in $D = \myh{\CC}\setminus E$ which is continuous in $\myh{\CC}$
and $v(z)\equiv\const$ on $E$ (cf. \cite{BuSu15}, \cite{Sue25}).
From identities \eqref{connect1} and \eqref{connect2} it follows that
\begin{equation}
\begin{aligned}
c_E(\theta)&=G^{\lambda_E(\theta)}_F(\infty)+\theta\gamma_E+G^{\lambda_F(\theta)}_E(\infty),\\
c_F(\theta)&=G^{\lambda_F(\theta)}_E(\infty)+\theta\gamma_E+(1+\theta)G^{\lambda_E(\theta)}_F(\infty)\\
&=c_E(\theta)+\theta G^{\lambda_E(\theta)}_F(\infty),\quad \gamma_E=-\log\mcap{E}=\log2.
\end{aligned}
\label{consts}
\end{equation}

In connection with the relation \eqref{balay} and the identities \eqref{connect1} and \eqref{connect2}, the potential-theoretic equilibrium problem \eqref{equivconj} can be considered as conjugated to the equilibrium problem \eqref{equiv}.

\section{Proof of Theorem \ref{the1}}\label{s3}

\subsection{}\label{s3s1}
Only for this Section \ref{s3}, and to simplify the notations,
we set $n=m$ and rewrite \eqref{hepa2} in the following form
\begin{align}
R_{n,1}(z):=
\bigl(Q_{2n}f-P_{2n,1}\bigr)(z)
&=O(z^{-n-1}),\quad z\to\infty,
\label{hepa2n1}\\
R_{n,2}(z):=
\bigl(Q_{2n}f^2-P_{2n,2}\bigr)(z)
&=O(z^{-n-1}),\quad z\to\infty,
\label{hepa2n2}
\end{align}
where $\mdeg{Q_{2n}},\mdeg{P_{2n,1}},\mdeg{P_{2n,2}}\leq{2n}$.

From \eqref{hepa2n1} it follows that for an arbitrary polynomial $p$, $\mdeg{p}\leq{n-1}$,
\begin{equation}
\oint_{\gamma_1}R_{n,1}(z)p(z)\,dz=0,
\label{gam1}
\end{equation}
where $\gamma_1$ is an arbitrary closed  curve which separates $E$ from $F$ and the infinity point $z=\infty$.
Thus from \eqref{gam1} and representation \eqref{repf} it follows that
\begin{equation}
\int_{-1}^1 Q_{2n}(x)p(x)\,d\sigma(x)=0.
\label{ort2f}
\end{equation}
In a similar way, based on \eqref{hepa2n2} and \eqref{repf2}, we obtain that for an arbitrary polynomial $q$, $\mdeg{q}\leq{n-1}$,
\begin{equation}
\int_{-1}^1 Q_{2n} (x)q(x)\biggl[\frac1{\sqrt{AB}}+\myh{\sigma}_2(x)\biggr]\,d\sigma(x)=0.
\label{ort2f2}
\end{equation}
From \eqref{ort2f} and \eqref{ort2f2} it follows that for arbitrary polynomials $p,q\in\PP_{n-1}$
\begin{equation}
\int_E Q_{2n}(x)[p(x)+q(x)\myh{\sigma}_2(x)]\,
d\sigma(x)=0.
\label{ort2m}
\end{equation}
Based on \eqref{ort2m} it is easy to obtain that $\mdeg{Q_{2n}}=2n$, all zero of $Q_{2n}$ are simple and belong to the open interval $(-1,1)$. Indeed, since polynomials $p,q\in\PP_{n-1}$ in \eqref{ort2m} are arbitrary, then for each polynomial $\omega_{2n-1}(z)=z^{2n-1}+\dotsb$,   $\mdeg{\omega_{2n-1}}=2n-1$, $Z(\omega_{2n-1})\subset E$, there exist two polynomials $p_{n-1},q_{n-1}\in\PP_{n-1}$ such that the function $p_{n-1}(z)+q_{n-1}(z)\myh{\sigma}_2(z)$ vanishes at all zeros of $\omega_{2n-1}$, i.e. the function 
\begin{equation}
\ell_n(z):=\frac{p_{n-1}(z)+q_{n-1}(z)\myh{\sigma}_2(z)}{\omega_{2n-1}(z)}
\label{defnelln}
\end{equation}
is a holomorphic function on $E$. Therefore, the function $\ell_n\in\HH(G)$ has a zero of multiplicity at least $n$ at the infinity point.  Thus we have that
\begin{equation}
\oint_{\gamma_2}\ell_n(z)z^k\,dz=0,
\quad k=0,\dots,n-2,
\label{ortelga2}
\end{equation}
where $\gamma_2$ is an arbitrary closed curve that separates the compact set $F$ from the infinity point.  From \eqref{ortelga2} we obtain the following orthogonality relations
\begin{equation}
\int_F \frac{q_{n-1}(y)y^k\,d\sigma_2(y)}{\omega_{2n-1}(y)}=0,
\quad k=0,\dots,n-2.
\label{ortqom}
\end{equation}
We have also that $\ell_n(z)s_n(z)\in\HH(G)$ for every polynomial $s_n\in\PP_{n-1}$ and thus
\begin{equation}
\ell_n(z)s_n(z)=\frac1{2\pi i}
\oint_{\gamma_2 }\frac{\ell_n(t)s_n(t)\,dt}{t-z}
=\int_F\frac{q_{n-1}(y)s_n(y)\,d\sigma_2(y)}{(z-y)\omega_{2n-1}(y)},
\quad z\in G.
\label{ells}
\end{equation}
Set $s_n=q_{n-1}$. Then by using \eqref{defnelln} we obtain a new form of \eqref{ells}
\begin{equation}
(p_{n-1}+q_{n-1}\myh{\sigma}_2)(z)
=\frac{\omega_{2n-1}(z)}{q_{n-1}(z)}
\int_F\frac{q_{n-1}^2(y)\,d\sigma_2(y)}{(z-y)\omega_{2n-1}(y)},
\quad z\in\Omega,
\label{pqsig}
\end{equation}
where the polynomial $q_{n-1}$ satisfies the orthogonality conditions \eqref{ortqom} which depend on the polynomial $\omega_{2n-1}$.
By the orthogonality relations \eqref{ort2m} with $p=p_{n-1}$ and $q=q_{n-1}$, the relation \eqref{pqsig} can be written in the following form
\begin{equation}
\int_E Q_{2n}(x)\frac{\omega_{2n-1}(x)}{q_{n-1}(x)}
\biggl\{
\int_F \frac{q_{n-1}^2(y)\,d\sigma_2(y)}{(x-y)\omega_{2n-1}(y)}
\biggr\}\,d\sigma(y)=0.
\label{ort2mn}
\end{equation}
Remark that \eqref{ort2mn} is valid for the given polynomial $Q_{2n}$ and for every polynomial $\omega_{2n-1}$, $\mdeg{\omega_{2n-1}}=2n-1$,  $Z(\omega_{2n-1})\subset E$.  Since $Z(\omega_{2n-1})\subset E$, then from \eqref{ortqom} it directly follows that all  $n-1$ zeros of $q_{n-1}$ are simple and belong to the compact set $F$, $Z(q_{n-1})\subset F$.  
From \eqref{ort2mn} by Remark \ref{rem23} it follows that $\mdeg{Q_{2n}}=2n$  and $Z(Q_{2n})\subset E$.
Further on, we suppose that these three polynomials $Q_{2n}$, $\omega_{2n-1}$ and $q_{n-1}$ are normalized  to be monic. 
Then we have that $Z(Q_{2n}),Z(\omega_{2n-1})\subset{E}$, $Z(q_{n-1})\subset{F}$ and $\mdeg{Q_{2n}}=2n$,  $\mdeg{\omega_{2n-1}}=2n-1$, $\mdeg{q_{n-1}}=n-1$. Thus, for a subsequence $\Lambda\subset\NN$ we have that as $n\to\infty$, $n\in\Lambda$,
\begin{equation}
\frac1{2n}\chi(Q_{2n}){\overset{*}\longrightarrow}\mu_Q,\quad
\frac1{2n-1}\chi(\omega_{2n-1}){\overset{*}\longrightarrow}\mu_\omega, \quad
\frac1n\chi(q_{n-1})
 \overset{*}\longrightarrow \mu_q,
\label{limpols}
\end{equation}
where $\mu_Q,\mu_\omega\in M_1(E)$ and $\mu_q\in M_1(F)$.
Since polynomials $\omega_{2n-1}(z)=z^{2n-1}+\dotsb$, $Z(\omega_{2n-1})\subset E$, are arbitrary,
then we can define them in such way that $\mu_\omega=\mu_Q$.

From \eqref{ort2mn} it follows by Corollary~\ref{cor4}, Corrolary~\ref{cor5},
Lemma~\ref{lem2} and Remark~\ref{rem24} that uniformly for $x\in E$
\begin{equation}
\begin{aligned}
\lim_{n\to\infty}\biggl|\int_F \frac{q_{n-1}^2(y)\,d\sigma_2(y)}{(x-y)\omega_{2n-1}(y)}\biggr|^{1/n}=e^{-c_1},\\
V^{\mu_q}(z)\equiv V^{\mu_Q}(z)-G_F^{\mu_Q}(z)+c_1, \quad z\in \myh\CC,\\
4V^{\mu_Q}(x)-V^{\mu_q}(x)\equiv c_2,\quad x\in E,
\end{aligned}
\label{relat2m}
\end{equation} 
$\supp{\mu_q}=F$, $\supp{\mu_ Q}=E$.
Thus from relations \eqref{relat2m} we obtain that the measure $\mu_Q\in M_1(E)$ satisfies the following identity
\begin{equation}
3V^{\mu_Q}(x)+G_F^{\mu_Q}(x)\equiv\const, 
\quad \supp{\mu_Q}=E,
\quad x\in E.
\label{iden1}
\end{equation}
From \cite[Section 7]{Sue25} and Section \ref{s2s3} follows that
the solution of the identity \eqref{iden1} is $\mu_Q=\lambda_E=\lambda_E(3)$,
it is unique, the constant is $\const=c_E(3)$, and $\supp{\lambda_E}=E$. 

Finally we obtain that as $n\to\infty$, $n\in\Lambda$,
\begin{equation}
\frac1{2n}\chi(Q_{2n}) \overset{*}\longrightarrow \lambda_E(3).
\label{limq2n}
\end{equation}

\subsection{}\label{s3s2}
From \eqref{hepa2n1} it follows that  $R_{n,1}(z)s_n(z)\in\HH(D)$, $R_{n,1}(\infty)s_n(\infty)=0$ for every polynomial $s_n\in\PP_{n-1}$ and thus
\begin{align}
R_{n,1}(z)s_n(z)&=\oint_{\gamma_1}
\frac1{2\pi i}\frac{Q_{2n}(t)f(t)s_n(t)\,dt}{t-z}
=\frac1{2\pi i}\int_E\frac{Q_{2n}(x)\Delta f(x)s_n(x)\,dx}{x-z}\nonumber\\
&=\int_E\frac{Q_{2n}(x)s_n(x)\,d\sigma(x)}{z-x},
\quad z\in D.
\label{rnsn}
\end{align}

Set $g(z):=(1/\sqrt{AB}+\myh{\sigma}_2)(z)\in\HH(\infty)$. Then from \eqref{hepa2n1} it follows that for every polynomial $p\in\PP_{n-1}$
\begin{equation}
\oint_\Gamma R_{n,1}(t)g(t)p(t)\,dt=0, 
\label{relGam}
\end{equation}
where $\Gamma$ is an arbitrary contour which separates the set $E\cup F$ from the infinity point. From \eqref{relGam} we obtain that
\begin{equation}
\oint_{\gamma_1}R_{n,1}(t)g(t)p(t)\,dt+
\oint_{\gamma_2}R_{n,1}(t)g(t)p(t)\,dt=0,
\label{relgam12}
\end{equation}
where $\gamma_1$ separates $E$ from $F$ and the infinity point and $\gamma_2$ separates $F$ from $E$ and the infinity point. From \eqref{relgam12} it directly follows that
\begin{equation}
\int_E Q_{2n}(x)\biggl[\frac1{\sqrt{AB}}+\myh{\sigma}_2(x)\biggr]p(x)\,d\sigma(x)+
\int_F R_{n,1}(y)p(y)\,d\sigma_2(y)=0
\label{relef}
\end{equation}
for every polynomial $p\in\PP_{n-1}$. By the orthogonality relation \eqref{ort2m} we have that the first integral in \eqref{relef} equals zero. Thus \eqref{relef} implies the following relation
\begin{equation}
\int_F R_{n,1}(y)p(y)\,d\sigma_2(y)=0
\label{relortF}
\end{equation}
for every polynomial $p\in\PP_{n-1}$. Since $\sigma'_2>0$ a.e. on $F$ then from \eqref{relortF} it follows that the function $R_{n,1}$ has at least $n$  simple zeros  on $F$. Let $\Omega_n(z)=z^n+\dotsb$  be the corresponding polynomial, i.e., the  the function $R_{n,1}/\Omega_n$ is a holomorphic function on $F$ and therefore, in $D$. Also we have that $R_{n,1}(\infty)q(\infty)/\Omega_n(\infty)=0$ for every polynomial $q\in\PP_{2n-1}$.  From the relation
\begin{equation}
\oint_{\gamma_1}\frac{R_{n,1}(t)q(t)}{\Omega_n(t)}\,dt=0
\label{relROm}
\end{equation}
follows that 
\begin{equation}
\int_E Q_{2n}(x)q(x)\frac{d\sigma(x)}{\Omega_n(x)}=0.
\label{relQOm}
\end{equation}
Since $Z(\Omega_n)\subset F$ and  $\mdeg{\Omega_n}=n$, then we can assume that for a subsequence $\Lambda\subset\NN$
$$
\frac1n\chi(\Omega_n) \overset{*}\longrightarrow {\mu_\Omega}\in M_1(F),
\quad n\to\infty, \quad n\in\Lambda.
$$
Thus by Lemma \ref{lem2} from \eqref{relQOm} it follows that
\begin{equation}
4V^{\lambda_E}(x)-V^{\mu_\Omega}(x)\equiv c_3=\const, 
\quad x\in E,
\label{equQOm}
\end{equation}
and by Corollary~\ref{cor5} we obtain that $c_3=3\gamma_E+c_E(3)$.
Since $R_{n,1}q/\Omega_n\in\HH(D)$ and  $R_{n,1}(\infty)q(\infty)/\Omega_n(\infty)=0$ for $q\in\PP_{2n-1}$, then for $z\in D$
\begin{equation}
\frac{R_{n,1}(z)q(z)}{\Omega_n(z)}
=\frac1{2 \pi i}
\oint_{\gamma_1}\frac{R_{n,1}(t)q(t)\,dt}
{(t-z)\Omega_n(t)}
=\int_E\frac{Q_{2n}(x)q(x)\,d\sigma(x)}
{(z-x)\Omega_n(x)}.
\label{repROm}
\end{equation}
Thus from \eqref{repROm}, \eqref{relQOm} and Remark \ref{rem25} it follows that for $z\in D$
\begin{equation}
R_{n,1}(z)=\frac{\Omega_n(z)}{q(z)}
\int_E \frac{Q_{2n}(x)q(x)\,d\sigma(x)}{(z-x)\Omega_n(x)}
=\frac{\Omega_n(z)}{Q_{2n}(z)}
\int_E\frac{Q_{2n}^2(x)\,d\sigma(x)}
{(z-x)\Omega_n(x)}.
\label{repR}
\end{equation}
From the representation \eqref{repR} and the orthogonality relations \eqref{relortF}, we obtain ultimately the following new orthogonality relations with respect to polynomial $\Omega_n(z)=z^n+\dotsb$
\begin{equation}
\int_F\Omega_n(y)p(y)\biggl\{\frac1{Q_{2n}(y)}
\int_E \frac{Q_{2n}^2(x)\,d\sigma(x)}{(y-x)\Omega_n(x)}\biggr\}\,d\sigma_2(y)=0,
\label{relOm}
\end{equation}
where $p\in\PP_{n-1}$ is an arbitrary polynomial. Set
\begin{equation}
\Psi_n(z):=\frac1{Q_{2n}(z)}
\int_E \frac{Q_{2n}^2(x)\,d\sigma(x)}{(z-x)\Omega_n(x)},
\label{repPsi}
\end{equation}
we obtain from \eqref{relOm} that
\begin{equation}
\int_F\Omega_n(y)p(y)\Psi_n(y)
\,d\sigma_2(y)=0,
\label{relOmPsi}
\end{equation}
$p\in\PP_{n-1}$. With that $\Psi_n$ we are in the situation  described in Lemma \ref{lem2}
(see Remark \ref{rem23}) for $n\in\Lambda$.
Thus by Lemma \ref{lem2} and Remark \ref{rem23} from \eqref{relOm} and \eqref{equQOm} it follows that
\begin{equation}
V^{\mu_\Omega}(y)-V^{\lambda_E}(y)\equiv c_4=\const,
\quad y\in F.
\label{equOmQ}
\end{equation}
Therefore  ${\mu_\Omega}=\beta_F(\lambda_E)$, $\lambda_F=\lambda_F(3)$ and $\Lambda=\NN$. Also we obtain that
\begin{equation}
V^{\mu_\Omega}(z)=V^{\lambda_E}(z)-G^{\lambda_E}_F(z)+c_4,
\quad
c_4=\int_E g_F(x,\infty)\,d\lambda_E(x),
\label{equOmQ2}
\end{equation}
and by Remark \ref{rem24} and relation \eqref{equQOm} we have uniformly in $D$ 
\begin{equation}
\biggl|\int_E \frac{Q_{2n}^2(x)\,d\sigma(x)}{(z-x)\Omega_n(x)}\biggr|^{1/n}\to e^{-c_3}.
\label{err}
\end{equation}

Finally, from the above results we  obtain that for $z\in D$
\begin{equation}
\(f-\frac{P_{2n,1}}{Q_{2n}}\)(z)
=\frac{R_{n,1}}{Q_{2n}}(z)=\frac{\Omega_n(z)}{Q_{2n}^2(z)}
\int_E\frac{Q_{2n}^2(x)\,d\sigma(x)}{(z-x)\Omega_n(x)}.
\label{ratcon1}
\end{equation}
Now from \eqref{ratcon1}, \eqref{equOmQ2}  and \eqref{err} we obtain that uniformly in $D$ and as $n\to\infty$
\begin{equation}
\biggl|\(f-\frac{P_{2n,1}}{Q_{2n}}\)(z)\biggr|^{1/n}
\to e^{-V^{\mu_\Omega}(z)+4V^{\lambda_E}(z)-c_3}=e^{3V{\lambda_E}(z)+G_F^{\lambda_E}(z)-c_4-c_3},
\label{ratcon2}
\end{equation}
since
$$
4V^{\lambda_E}(z)-V^{\mu_\Omega}(z)=
4V^{\lambda_E}(z)-
V^{\lambda_E}(z)+G^{\lambda_E}_F(z)-c_4
=3V^{\lambda_E}(z)+G^{\lambda_E}_F(z)-c_4.
$$
Thus, by this identity (see Section \eqref{s2s3})
$$
3V^{\lambda_E}(z)+G^{\lambda_E}_F(z)-c_E(3)\equiv -G^{\lambda_F}_E(z)-3g_E(z,\infty),
$$
and since $c_3=3\gamma_E+c_E(3)$ and $c_3+c_4=c_E(3)$ (see \eqref{consts}),  
from \eqref{ratcon2} we obtain that
\begin{align}
\biggl|\(f-\frac{P_{2n,1}}{Q_{2n}}\)(z)\biggr|^{1/n}
 &\to  e^{3V^{\lambda_E(3)}(z)+G^{\lambda_E(3)}_F(z)-c_4-c_3}\nonumber\\
&=e^{-G^{\lambda_F(3)}_E(z)-3g_E(z,\infty)+c_E(3)-c_4-c_3}\nonumber\\
&=e^{-G^{\lambda_F(3)}_E(z)-3g_E(z,\infty)}<1.
\label{ratcon3}
\end{align}
From \eqref{ratcon3} the limit relation \eqref{hp2conv}  follows immediately. Theorem \ref{the1} is proved.

\section{Proof of Theorem \ref{the2}}\label{s4}

\subsection{}\label{s4s1}
Only for this Section \ref{s4}, and to simplify the notations,
we set $n=\ell$ and rewrite \eqref{hepa3} in the following form
\begin{align}
R_{n,1}(z):=
\bigl(Q_{3n}f-P_{3n,1}\bigr)(z)
&=O(z^{-n-1}),\quad z\to\infty,
\label{hepa3n1}\\
R_{n,2}(z):=
\bigl(Q_{3n}f^2-P_{3n,2}\bigr)(z)
&=O(z^{-n-1}),\quad z\to\infty,
\label{hepa3n2}\\
R_{n,3}(z):=
\bigl(Q_{3n}f^3-P_{3n,3}\bigr)(z)
&=O(z^{-n-1}),\quad z\to\infty,
\label{hepa3n3}
\end{align}
where $\mdeg{Q_{3n}},\mdeg{P_{3n,1}},\mdeg{P_{3n,2}},\mdeg{P_{3n,3}}\leq{3n}$.
From \eqref{hepa3n1}--\eqref{hepa3n3} and \eqref{repf}--\eqref{repf3} it follows that (cf. \eqref{ort2f}, \eqref{ort2f2}) 
\begin{align}
&\int_E Q_{3n}(x)p(x)\,d\sigma(x)=0,\label{ort3f}\\
&\int_E Q_{3n}(x)q(x)\biggl[\frac1{\sqrt{AB}}+\myh{\sigma}_2(x)\biggr]\,d\sigma(x)=0,\label{ort3f2}\\
&\int_E Q_{3n}(x)r(x)\biggl[\frac1{AB}+\frac1{\sqrt{AB}}\myh{\sigma}_2(x)+\myh{\<\sigma_2,\sigma\>}(x)\biggr]\,d\sigma(x)=0\label{ort3f3}
\end{align}
for every polynomials $p,q,r\in\PP_{n-1}$. After combining \eqref{ort3f}--\eqref{ort3f3} we obtain the main orthogonality relation on $E$
\begin{equation}
\int_E Q_{3n}(x)\biggl\{
p(x)+q(x)\myh{\sigma}_2(x)+r(x)\myh{\<\sigma_2,\sigma\>}(x)\biggr\}\,d\sigma(x)=0
\label{ort3m}
\end{equation}
for every polynomials $p,q,r\in\PP_{n-1}$. Since all three polynomials $p,q,r\in\PP_{n-1}$ are free, and the total number of free parameters (i.e. the coefficients of $p,q,r$) equals $3n$, then for each  polynomial $\omega_{3n-1}\in\PP_{3n-1}$, $Z(\omega_{3n-1})\subset E$, there exist polynomials $p_{n-1},q_{n-1},r_{n-1}\in\PP_{n-1}$ such that the function
\begin{equation}
\ell_n(z):=\frac{
p_{n-1}(z)+q_{n-1}(z)\myh{\sigma}_2(z)+r_{n-1}(z)\myh{\<\sigma_2,\sigma\>}(z)}{\omega_{3n-1}(z)}
\label{form3l}
\end{equation}
is a holomorphic function on $E$ (and thus outside of $F$). Indeed, given a polynomial $\omega_{3n-1}\in\PP_{3n-1}$, $Z(\omega_{3n-1})\subset E$, the corresponding polynomials $p_{n-1}$, $q_{n-1}$, $r_{n-1}\in\PP_{n-1}$ can be constructed in such way that the function 
$$
p_{n-1}(z)+q_{n-1}(z)\myh{\sigma}_2(z)+r_{n-1}(z)\myh{\<\sigma_2,\sigma\>}(z)
$$
vanishes at all $3n-1$ zeros of the polynomial $\omega_{3n-1}$. Since we have $3n$ free parameters to solve this problem of linear interpolation, a solution of this problem always exists.

From the representation \eqref{form3l} it follows that the function $\ell_n(z)$  has a zero at the infinity point of multiplicity at least $2n$. Thus
\begin{equation}
\int_{\gamma_2} \ell_n(z)z^k\,dz=0,
\quad k=0,\dots,2n-2,
\label{ort3l1}
\end{equation}
where $\gamma_2$ is an arbitrary contour which separates the compact set $F$ from $E$ and the infinity point.
As already demonstrated, the relation \eqref{ort3l1} implies that
\begin{equation}
\int_F \frac{q_{n-1}(y)+r_{n-1}(y)\myh{\sigma}(y)}{\omega_{3n-1}(y)}y^k\,d\sigma_2(y)=0,\quad k=0,\dots,2n-2.
\label{ort3l2}
\end{equation}
Since $Z(\omega_{3n-1})\subset E$, from \eqref{ort3l2} it directly follows that the function $L_n(z):=q_{n-1}(z)+r_{n-1}(z)\myh{\sigma}(z)$  has at least  $2n-1$ simple zeros on the set $F$. Let us denote the corresponding polynomial by $\Omega_{2n-1}(z)=z^{2n-1}+\dotsb$, i.e.  the function
\begin{equation}
L_n(z):=\frac{q_{n-1}(z)+r_{n-1}(z)\myh{\sigma}(z)}{\Omega_{2n-1}(z)}
\label{form3l2}
\end{equation}
is a holomorphic function on $F$ (and therefore, in $\myh\CC\setminus E$). Since the function $L_n(z)$ has a zero at the infinity point of the multiplicity at least $n$, then by the established way we obtain the following orthogonality relations on $E$
\begin{equation}
\int_E  \frac{r_{n-1}(x)x^k\,d\sigma(x)}{\Omega_{2n-1}(x)},
\quad k=0,\dots,n-2,
\label{ort3L1}
\end{equation}
where $Z(\Omega_{2n-1})\subset F$. From \eqref{ort3L1} it follows that $\mdeg{r_{n-1}}=n-1$ and $Z(r_{n-1})\subset E$.

Suppose that as $n\to\infty$, $n\in\Lambda \subset\NN$, 
\begin{equation}
\frac1{2n}\chi(\Omega_{2n-1}) \overset{*}\longrightarrow \mu_\Omega\in M_1(F).
\label{limdis1}
\end{equation}
Then by Lemma \ref{lem2} and Corollary \ref{cor5} we have that
for $r_{n-1}(z)=z^{n-1}+\dotsb$ as $n\to\infty$, $n\in\Lambda \subset\NN$, 
\begin{equation}
\frac1n\chi(r_{n-1}){\overset{*}\longrightarrow}\mu_r\in M_1(E),
\quad \mu_r=\beta_E(\mu_\Omega), \quad
\supp{\mu_r}=E.
\label{limdis2}
\end{equation}
Therefore 
\begin{align*}
&V^{\mu_r}(x)-V^{\mu_\Omega}(x)\equiv\const, && x\in E, \\
&V^{\mu_r}(z)=V^{\mu_\Omega}(z)-G_E^{\mu_\Omega}(z)+\const, && \const=\int_F g_E(y,\infty)\,d\mu_\Omega(y).
\end{align*}
By the Cauchy formula and based on \eqref{ort3L1} we obtain that
\begin{align}
L_n(z)&=\frac1{2\pi i} \int_{\gamma_1}
\frac{L_n(t)}{t-z}dt=
\int_E\frac{r_{n-1}(x)\,d\sigma(x)}{(z-x)\Omega_{2n-1}(x)}\nonumber\\
&=\frac1{r_{n-1}(z)}\int_E \frac{r^2(x)\,d\sigma(x)}{(z-x)\Omega_{2n-1}(x)}.\label{repL1}
\end{align}
From \eqref{repL1} it follows that
\begin{equation}
q_{n-1}(z)+r_{n-1}(z)\myh{\sigma}(z)
=\frac{\Omega_{2n-1}}{r_{n-1}(z)}
\int_E \frac{r^2_{n-1}(x)\,d\sigma(x)}{(z-x)\Omega_{2n-1}(x)},\quad z\in D.
\label{repL2}
\end{equation}
Ultimately from \eqref{ort3l2} and  \eqref{repL2} we obtain the following orthogonality conditions on $F$
\begin{equation}
\int_F\frac{\Omega_{2n-1}(y)}{\omega_{3n-1}(y)r_{n-1}(y)}s(y)\biggl\{\int_E\frac{r^2_{n-1}(x)\,d\sigma(x)}{(y-x)\Omega_{2n-1}(x)}\biggr\}\,d\sigma_2(y)=0
\label{ort3l3}
\end{equation}
for every polynomial $s\in\PP_{2n-2}$.

By Lemma \ref{lem1} and Lemma \ref{lem2} (see Corollary \ref{cor5} and Remark \ref{rem23}) it follows that if
$$
\frac1{3n}\chi(\omega_{3n-1}) \overset{*}\longrightarrow \mu_\omega\in M_1(E)
$$
as $n\to\infty$, $n\in\Lambda'\subset \Lambda$ (recall that $Z(\omega_{3n-1})\subset E$), then
$$
\frac1{2n}\chi(\Omega_{2n-1}){\overset{*}\longrightarrow}\mu_\Omega\in M_1(F)
$$
where
$$
-4V^{\mu_\Omega}(y)+3V^{\mu_\omega}(y)+V^{\mu_r}(y)\equiv\const,\quad y\in F
$$
(cf. \eqref{limdis1}). Finally from the above we obtain that
$$
4V^{\mu_\Omega}(y)-3V^{\mu_\omega}(y)-V^{\mu_r}(y)
=3V^{\mu_\Omega}(y)+G_E^{\mu_\Omega}(y)-3V^{\mu_\omega}(y)\equiv\const,\quad y\in F.
$$
From the above and the representation \eqref{form3l} for the function $\ell_n(z)$ we obtain the following result 
\begin{align}
\ell_n(z)s(z)&=\int_F
\frac{(q_{n-1}(y)+r_{n-1}(y)\myh{\sigma}(y))s(y)
\,d\sigma_2(y)}{\omega_{3n-1}(y)(z-y)}\nonumber\\
&=\int_F\frac{\Omega_{2n-1}(y)s(y)}{r_{n-1}(y)\omega_{3n-1}(y)}\frac1{(z-y)}
\biggl\{
\int_E\frac{r^2_{n-1}(x)\,d\sigma(x)}{(y-x)\Omega_{2n-1}(x)}\biggr\}\,d\sigma_2(y),
\label{form3l3}
\end{align}
where $s\in\PP_{2n-1}$ is an arbitrary polynomial.
Set $s=\Omega_{2n-1}$. Therefore, based on \eqref{form3l3} we obtain that
\begin{align}
&p_{n-1}(z)+q_{n-1}(z)\myh{\sigma}_2(z)+r_{n-1}(z)\myh{\<\sigma_2,\sigma\>}(z)\nonumber\\
&=\frac{\omega_{3n-1}(z)}{\Omega_{2n-1}(z)}\int_F \frac{\Omega^2_{2n-1}(y)}{r_{n-1}(y)\omega_{3n-1}(y)}\frac1{(z-y)}\biggl\{\int_E\frac{r^2_{n-1}(x)\,d\sigma(x)}{(y-x)\Omega_{2n-1}(x)}\biggr\}\,d\sigma_2(y).\label{repform3m}
\end{align}
Since $\omega_{3n-1}$ is an arbitrary polynomial, $Z(\omega_{3n-1})\subset E$, then from the orthogonality relations (cf. \eqref{ort3m}) on $E$
\begin{equation}
\int_E Q_{3n}(x)\{
p_{n-1}(x)+q_{n-1}(x)\myh{\sigma}_2(x)
+r_{n-1}(x)\myh{\<\sigma_2,\sigma\>}(x)
\}\,d\sigma(x)=0
\label{ort3mn}
\end{equation}
and from \eqref{repform3m} it follows that
\begin{align}
&\int_E Q_{3n}(x)\omega_{3n-1}(x)\Psi_n(x)\,d\sigma(x)=0,\label{ort3mnf}\\
&\Psi_n(x):=\frac1{\Omega_{2n-1}(x)}\int_F \frac{\Omega^2_{2n-1}(y)}{r_{n-1}(y)\omega_{3n-1}(y)}\frac1{(x-y)}\biggl\{\int_E\frac{r^2_{n-1}(t)\,d\sigma(t)}{(y-t)\Omega_{2n-1}(t)}\biggr\}\,d\sigma_2(y).\nonumber
\end{align}
Since $\Psi_n(x)>0$ then from \eqref{ort3mnf} it follows that $\mdeg{Q_{3n}}=3n$ and $Z(Q_{3n})\subset E$.

Let for a subsequence  $\Lambda''\subset\Lambda'\subset\Lambda$
\begin{equation}
\frac1{3n}\chi(Q_{3n}){\overset{*}\longrightarrow}\mu_Q\in M_1(E), \quad n\to\infty, \quad n\in\Lambda''.
\label{limdisQ}
\end{equation}
Then by Lemma \ref{lem2} and Corollary \ref{cor5} we have that
\begin{equation}
3V^{\mu_Q}(x)-V^{\mu_\Omega}(x)\equiv\const,
\quad x\in E,\quad\supp{\mu_Q}=E.
\label{relequQOm}
\end{equation}
Also by Lemma \ref{lem2}, Corollary \ref{cor4} and \eqref{limdis2}
\begin{align*}
&V^{\mu_r}(x)-V^{\mu_\Omega}(x)\equiv\const, && x \in E, \\
&4V^{\mu_\Omega}(y)-3V^{\mu_\omega}(y)-V^{\mu_r}(y)\equiv\const, && y\in F,
\end{align*}
$\supp{\mu_r}=E$, $\supp{\mu_\Omega}= F$, $\supp{\mu_\omega}= E$. Since
$$V^{\mu_r}(z)=V^{\mu_\Omega}(z)-G_E^{\mu_\Omega}(z)+\const, \quad z\in\myh{\CC}, $$
then ultimately
\begin{align}
&3V^{\mu_Q}(x)-V^{\mu_r}(x)\equiv\const, && x\in E,\label{equQr1}\\
&3V^{\mu_\Omega}(y)+G_E^{\mu_\Omega}(y)-3V^{\mu_\omega}(y)\equiv\const, && y\in F.\label{equOmom1}
\end{align}
Set
$$
v(z):=3V^{\mu_Q}(z)-V^{\mu_r}(z)+2g_E(z,\infty).
$$
Then from \eqref{equQr1} it follows that $v(z)\equiv2\gamma_E=\const$, $z\in\myh{\CC}$, and thus
\begin{equation}
3V^{\mu_Q}(x)-V^{\mu_r}(x)\equiv2\gamma_E,\quad x\in E.
\label{equQr2}
\end{equation}

\subsection{}\label{s4s2}
Let us multiply both sides of \eqref{hepa3n1} by $(1/\sqrt{AB}+\myh{\sigma}_2(z))q(z)$, where $q\in\PP_{n-1}$, to obtain
\begin{equation}
\int_\Gamma R_{n,1}(t)\biggl(\frac1{\sqrt{AB}}+\myh{\sigma}_2(t)\biggr)q(t)\,dt=0,
\label{intR11}
\end{equation}
where $\Gamma$ is an arbitrary contour which separates the set $E\cup F$ from the infinity point.
The relation \eqref{intR11} is equivalent to the following
\begin{equation}
\int_{\gamma_1}R_{n,1}(t)\biggl(\frac1{\sqrt{AB}}+\myh{\sigma}_2(t)\biggr)q(t)\,dt+
\int_{\gamma_2} R_{n,1}(t)\biggl(\frac1{\sqrt{AB}}+\myh{\sigma}_2(t)\biggr)q(t)\,dt=0,
\label{intR12}
\end{equation}
where $\gamma_1$ and $\gamma_2$  are the contours which separates the compact sets $E$ and $F$ from each other and from the infinity point. From \eqref{intR12} we obtain in the established way the following relation
\begin{equation}
\int_E Q_{3n}(x)\biggl(\frac1 {\sqrt{AB}}+\myh{\sigma}_2(x)\biggr)q(x)\,d\sigma(x)
+\int_F R_{n,1}(y)q(y)\,d\sigma_2(y)=0
\label{intR13}
\end{equation}
for every polynomial $q\in\PP_{n-1}$. By the orthogonality relations \eqref{ort3f2} the first integral in \eqref{intR13} equals zero and we obtain that
\begin{equation}
\int_F R_{n,1}(y)q(y)\,d\sigma_2(y)=0
\label{intR14}
\end{equation}
for every $q\in\PP_{n-1}$.

Let us now multiply both sides of \eqref{hepa3n1} with the function
$$
\biggl\{\frac1{AB}+\frac1{\sqrt{AB}}\myh{\sigma}_2(z)+\myh{\<\sigma_2,\sigma\>}(z)\biggr\}r(z),
\quad r\in\PP_{n-1},
$$
to obtain in the similar way the next orthogonality relation for $R_{n,1}$ on $F$
\begin{equation}
\int_F R_{n,1}(y)\biggl\{\frac1{\sqrt{AB}}+\myh{\sigma}(y)\biggr\}r(y)\,d\sigma_2(y),
\quad r\in\PP_{n-1}.
\label{intR15}
\end{equation}
By combining the relation \eqref{intR14} and \eqref{intR15} we obtain finally the main orthogonality relation for $R_{n,1}$  on $F$ (cf. \eqref{ort3l2}):
\begin{equation}
\int_F R_{n,1}(y)(q(y)+r(y)\myh{\sigma}(y))
\,d\sigma_2(y)=0
\label{intR1m}
\end{equation}
for every polynomials $q,r\in\PP_{n-1}$. Since the polynomials $q,r\in\PP_{n-1}$ are arbitrary, then they can be chosen in such way that the function $(q+r\myh \sigma)(z)$ has $2n-1$ zeros in arbitrary predetermined set of $2n-1$ points on $F$ (cf. \eqref{form3l2}). Therefore from \eqref{intR1m} it follows that the function $R_{n,1}$ has at least $2n$ simple zeros on $F$. Let us denote by $u_{2n}(z)=z^{2n}+\dotsb$ the corresponding polynomial. Then the function $R_{n,1}/u_{2n}$ is a holomorphic function on $F$ and therefore it is holomorphic in  $\myh{\CC}\setminus E$. In addition, the function $R_{n,1}/u_{2n}$ has zero at the infinity point of order at least  $3n+1$. Let $s(z)\in\PP_{3n-1}$ be an arbitrary polynomial. Then 
$$
\int_{\gamma_1}\frac{R_{n,1}(t)}{u_{2n}(t)}s(t)\,dt=0.
$$
From the last relation by the already known way we obtain the following orthogonality relations on $E$
\begin{equation}
\int_E  Q_{3n}(x)\frac{s(x)}{u_{2n}(x)}\,d\sigma(x)=0
\label{ortQu}
\end{equation}
for every polynomial $s\in\PP_{3n-1}$. 

Since  $Z(u_{2n})\subset F$, then there exists a subsequence $\Lambda'''\subset \Lambda''\subset \Lambda'\subset \Lambda$ such that
\begin{equation}
\frac1{2n}\chi(u_{2n}){\overset{*}\longrightarrow}\mu_u\in M_1(F),
\quad n\to\infty,\quad n\in\Lambda'''.
\label{limdisu}
\end{equation}
Since by \eqref{limdisQ} $\frac1{3n}\chi(Q_{3n}){\overset{*}\longrightarrow}\mu_Q$, then by Lemma \ref{lem2} we obtain from \eqref{ortQu} the identity
\begin{equation}
3V^{\mu_Q}(x)-V^{\mu_u}(x)\equiv\const,
\quad x\in E.
\label{equQu}
\end{equation}
By Cauchy formula we obtain that
\begin{equation}
\frac{(R_{n,1}s)(z)}{u_{2n}(z)}
=\frac1{2\pi i}
\int_{\gamma_1}\frac{R_{n,1}(t)s(t)\,dt}
{(t-z)u_{2n}(x)}
=\int_E\frac{Q_{3n}(x)s(x)\,d\sigma(x)}{(z-x)u_{2n}(x)}
\label{repRu}
\end{equation}
and thus (see Remark \ref{rem25})
\begin{equation}
R_{n,1}(z)=\frac{u_{2n}(z)}{Q_{3n}(z)}
\int_E\frac{Q^2_{3n}(x)\,d\sigma(x)}{(z-x)u_{2n}(x)},
\quad z\in D.
\label{repRuQ}
\end{equation}
Since in the representation $q(z)+r(z)\myh{\sigma}(z)$ both polynomials $q,r\in\PP_{n-1}$ are free, then for each predetermined polynomial $v_{2n-1}(z)=z^{2n-1}+\dotsb$, $v_{2n-1}\in\PP_{2n-1}$, $Z(v_{2n-1})\subset  F$, there exist two polynomials $\myt{q}_{n-1},\myt{r}_{n-1}\in\PP_{n-1}$ such that the function $(\myt{q}_{n-1}(z)+\myt{r}_{n-1}(z)\myh{\sigma}(z))/v_{2n-1}(z)$ is a holomorphic function in $\myh{\CC}\setminus E$. Then
\begin{align}
&\myt{q}_{n-1}(z)+\myt{r}_{n-1}(z)\myh{\sigma}(z)=
\frac{v_{2n-1}(z)}{\myt{q}_{n-1}(z)}
\int_E \frac{\myt{r}^2_{n-1}(x)\,d\sigma(x)}{(z-x)v_{2n-1}(x)},
\label{repqrv}\\
&\int_E \frac{\myt{r}_{n-1}(x)x^k\,d\sigma(x)}{v_{2n-1}(x)}=0,
\quad k=0,\dots,n-2.
\label{ortrv}
\end{align}
Finally from \eqref{intR1m}, \eqref{repRuQ} and \eqref{repqrv} we obtain the following orthogonal relation on $F$ for the polynomial $u_{2n}\in\PP_{2n}$  with respect to  every polynomial $v_{2n-1}\in\PP_{n-1}$ 
\begin{equation}
\int_F\frac{u_{2n}(y)}{Q_{3n}(y)}
\biggl\{
\int_E\frac{Q^2_{3n}(x)\,d\sigma(x)}{(y-x)u_{2n}(x)}
\biggr\}\frac{v_{2n}(y)}{\myt{r}_{n-1}(y)}
\biggl\{\int_E\frac{\myt{r}^2_{n-1}(t)\,d\sigma(t)}{(y-t)v_{2n-1}(t)}\biggr\}\,d\sigma_2(y)=0.
\label{ortuQ}
\end{equation}
Let introduce a varying weight function $\Psi_n(z)$
\begin{equation}
\Psi_n(y)=
\frac1{Q_{3n}(y)}
\biggl\{
\int_E\frac{Q^2_{3n}(x)\,d\sigma(x)}{(y-x)u_{2n}(x)}
\biggr\}\frac1{\myt{r}_{n-1}(y)}
\biggl\{\int_E\frac{\myt{r}^2_{n-1}(x)\,d\sigma(x)}{(y-x)v_{2n-1}(x)}\biggr\}.
\label{repPsi2}
\end{equation}
Then the relation \eqref{ortuQ} can be written as
\begin{equation}
\int_F u_{2n}(y)v_{2n-1}(y)\Psi_n(y)\,d\sigma_2(y)=0,
\label{ortuvPsi}
\end{equation}
where $u_{2n}\in\PP_{2n}$, $Z(u_{2n})\subset F$, $u_{2n}$ is fixed, and $v_{2n-1}\in\PP_{2n-1}$ is an arbitrary polynomial, $Z(v_{2n-1})\subset F$.

Ultimately by Lemma \ref{lem2}, Corollary \ref{cor5}, Remark \ref{rem24}, the relations \eqref{equQu}, \eqref{ortuvPsi}, \eqref{repPsi2} and \eqref{ortQu} we have the following identities for $\mu_Q,\mu_{\myt{r}}\in M_1(E)$, $\mu_u\in M_1(F)$
\begin{align}
&3V^{\mu_Q}(x)-V^{\mu_u}(x)\equiv\const, && x\in E, && \supp{\mu_Q}=E, \label{equQu2}\\
&V^{\mu_{\myt{r}}}(x)-V^{\mu_u}(x)\equiv\const, && x\in E, && \supp{\mu_{\myt{r}}}=E,\label{equru}\\
&4V^{\mu_u}(y)-3V^{\mu_Q}(y)-V^{\mu_{\myt{r}}}(y)\equiv\const, && y\in F, && \supp{\mu_u}=F.\label{equuQr}
\end{align}
From \eqref{equQu2} and \eqref{equru} it follows that 
\begin{equation}
\begin{aligned}
3V^{\mu_Q}(z)&\equiv V^{\mu_u}(z)-G_E^{\mu_u}(z)-2g_E(z,\infty)+\const, \\
V^{\mu_{\myt{r}}}(z)&\equiv V^{\mu_u}(z)-G_E^{\mu_u}(z)+\const, \quad z\in\myh\CC.
\end{aligned}
\label{equQur}
\end{equation}
By substituting \eqref{equQur} in \eqref{equuQr} we obtain that
$$
-2V^{\mu_u}(y)-2g_E(y,\infty)-2G_E^{\mu_u}(y)\equiv\const,\quad y\in F.
$$
Thus
$$
V^{\mu_u}(y)+G_E^{\mu_u}(y)+g_E(y,\infty)\equiv\const,\quad y\in F,
$$
and therefore by \eqref{equivconj} we have $\mu_u=\lambda_F(1)$ with $\const=c_F(1)$.

Finally from \eqref{repRuQ} we obtain (see Remark \ref{rem25})
\begin{equation}
\(f-\frac{P_{3n,1}}{Q_{3n}}\)(z)=\frac{u_{2n}(z)}{Q^2_{3n}(z)}
\int_E\frac{Q^2_{3n}(x)\,d\sigma(x)}{(z-x)u_{2n}(x)},
\quad z\in D.
\label{repRPQ}
\end{equation}
Set
\begin{equation}
v(z):=3V^{\mu_Q}(z)-V^{\lambda_F(1)}(z)+G_E^{\lambda_F(1)}(z)+2g_E(z,\infty).
\label{repQla}
\end{equation}
From \eqref{equQu2} it follows that
$$
v(x)\equiv\const,\quad x\in E.
$$
Also it is easy to see that $v(z)$ is a continuous function in $\myh{\CC}$ and is a harmonic function in $\myh{\CC}\setminus E$.
Therefore $v(z)\equiv\const$ in $\myh{\CC}$, that is $v(z) \equiv v(\infty)$. From \eqref{repQla} we obtain that
\begin{align}
v(z)&=3V^{\mu_Q}(z)-V^{\lambda_F(1)}(z)+G_E^{\lambda_F(1)}(z)+2g_E(z,\infty)\nonumber\\
&\equiv G_E^{\lambda_F(1)}(\infty)+2\gamma_E=v(\infty).
\label{repQlab}
\end{align}
Set $c_{E,F} = G_E^{\lambda_F(1)}(\infty)+2\gamma_E$. Thus
\begin{equation}
3V^{\mu_Q}(z)-V^{\lambda_F(1)}(z)\equiv -G_E^{\lambda_F(1)}(z)-2g_E(z,\infty)+c_{E,F}, 
\quad z\in\myh{\CC}\setminus E,
\label{rate1}
\end{equation}
and
\begin{equation}
3V^{\mu_Q}(x)-V^{\lambda_F(1)}(x)\equiv c_{E,F},
\quad 
x\in E.
\label{rate2}
\end{equation}
Therefore by Remark \ref{rem24} we have that
\begin{equation}
\lim_{\stackrel{n\to\infty}{n\in\Lambda}}\biggl|
\int_E\frac{Q^2_{3n}(x)\,d\sigma(x)}{(z-x)u_{2n}(x)}
\biggr|^{1/2n}=e^{-c_{E,F}},
\quad z\in D.
\label{rate3}
\end{equation}
Finally we obtain from \eqref{repRPQ} and \eqref{rate1} that
$$
\lim_{\stackrel{n\to\infty}{n\in\Lambda}}\biggl|
\(f-\frac{P_{3n,1}}{Q_{3n}}\)\biggr|^{1/2n}
=e^{3V^{\mu_Q}(z)-V^{\lambda_F(1)}(z)}e^{-c_{E,F}}
=e^{-G_E^{\lambda_F(1)}(z)-2g_E(z,\infty)}<1.
$$
Theorem \ref{the2} is proved.

\end{document}